\documentclass[11pt, reqno]{amsart}

\usepackage[top=3.75cm, bottom=3cm, left=3cm, right=3cm]{geometry}
\usepackage{amsmath}
\usepackage{amsthm}
\usepackage{amsfonts}
\usepackage{amssymb}
\usepackage{mathtools}
\usepackage{bm}
\usepackage[colorlinks=true, citecolor=darkblue, urlcolor=darkblue, linkcolor=darkblue, linktocpage = true]{hyperref}
\usepackage{color}
\usepackage[all,cmtip]{xy}
\usepackage{enumitem}
\usepackage{tikz}  
\usepackage{graphicx}
\usepackage{scalerel}
\usepackage{comment}
\usepackage[mathscr]{eucal}
\usepackage{stmaryrd}
\usepackage{tikz-cd}

\definecolor{darkblue}{rgb}{0,0,.75}

\numberwithin{equation}{section}

\theoremstyle{plain}
\newtheorem{theorem}{Theorem}[section]
\newtheorem{lemma}[theorem]{Lemma}
\newtheorem{proposition}[theorem]{Proposition}
\newtheorem{corollary}[theorem]{Corollary}

\theoremstyle{definition}

\newtheorem{definition}[theorem]{Definition}

\newtheorem{remark}[theorem]{Remark}
\newtheorem{notation}[theorem]{Notation}
\newtheorem{warning}[theorem]{Warning}
\newtheorem{construction}[theorem]{Construction}

\makeatletter
\newtheoremstyle{italicsname}
 {3pt}
 {3pt}
 {\itshape}
 {}
{\bf}
 {.}
 {.5em}
 {\thmname{#1}\thmnumber{\@ifnotempty{#1}{ }#2}%
 \thmnote{ {\the\thm@notefont(#3)}}}
\makeatother
\theoremstyle{italicsname}

\setlist[itemize]{topsep=5pt,itemsep=3pt}
\setlist[enumerate]{topsep=5pt,itemsep=3pt}

\newcommand{\set}[1]{\left\{ \, #1 \, \right\}}

\newcommand{\Dperf}{\mathrm{D}_{\mathrm{perf}}}
\newcommand{\Dqc}{\mathrm{D}_{\mathrm{qc}}}

\DeclareMathOperator{\Forg}{Forg}
\DeclareMathOperator{\Inf}{Inf}

\DeclareMathOperator{\Hdg}{Hdg}

\DeclareMathOperator{\colim}{colim}

\newcommand{\Set}{\mathrm{Set}}

\newcommand{\Sch}{\mathrm{Sch}}

\newcommand{\Grpd}{\mathrm{Grpd}}

\DeclareMathOperator{\Pic}{Pic}
\newcommand{\cPic}{\mathcal{P}\!{\it ic}}

\newcommand{\wtilde}{\widetilde}

\DeclareMathOperator{\Spec}{Spec}

\newcommand{\op}{\mathrm{op}}

\newcommand{\cHom}{\mathcal{H}\!{\it om}}
\DeclareMathOperator{\sh}{sh}

\DeclareMathOperator{\Hom}{Hom}

\DeclareMathOperator{\Ext}{Ext}
\DeclareMathOperator{\Aut}{Aut}
\DeclareMathOperator{\cAut}{\mathcal{A}{\it ut}}

\DeclareMathOperator{\HH}{HH}

\DeclareMathOperator{\cExt}{\mathcal{E}\!{\it xt}}
\DeclareMathOperator{\cHH}{\mathcal{HH}}

\newcommand{\id}{\mathrm{id}}

\newcommand{\ob}{\mathrm{ob}}

\DeclareMathOperator{\Br}{Br}

\DeclareMathOperator{\CH}{CH}

\newcommand{\rtop}{\mathrm{top}}

\newcommand{\an}{\mathrm{an}}

\newcommand{\Ktop}[1][]{\rK_{#1}^{\rtop}}

\DeclareMathOperator{\NS}{NS}

\newcommand{\et}{\mathrm{\acute{e}t}}
\newcommand{\tors}{\mathrm{tors}}

\newcommand{\ev}{\mathrm{ev}}

\newcommand{\bmu}{\bm{\mu}}

\DeclareMathOperator{\ch}{{ch}}

\newcommand{\cO}{\mathcal{O}}

\newcommand{\cC}{\mathscr{C}}
\newcommand{\cD}{\mathscr{D}}

\newcommand{\cM}{\mathcal{M}}
\newcommand{\cN}{\mathcal{N}}

\newcommand{\cX}{\mathcal{X}}
\newcommand{\cY}{\mathcal{Y}}

\newcommand{\rH}{\mathrm{H}}

\newcommand{\rK}{\mathrm{K}}

\newcommand{\rR}{\mathrm{R}}

\newcommand{\fp}{\mathfrak{p}}

\newcommand{\fm}{\mathfrak{m}}

\newcommand{\bC}{\mathbf{C}}

\newcommand{\bG}{\mathbf{G}}

\newcommand{\bZ}{\mathbf{Z}}

\newcommand{\bQ}{\mathbf{Q}}

\usepackage{xpatch}
\makeatletter   
\xpatchcmd{\@tocline}
{\hfil\hbox to\@pnumwidth{\@tocpagenum{#7}}\par}
{\ifnum#1<0\hfill\else\dotfill\fi\hbox to\@pnumwidth{\@tocpagenum{#7}}\par}
{}{}
\makeatother 

\makeatletter
\renewcommand\part{%
   \if@noskipsec \leavevmode \fi
   \par
   \addvspace{4ex}%
   \@afterindentfalse
   \secdef\@part\@spart}

\def\@part[#1]#2{%
    \ifnum \c@secnumdepth >\m@ne
      \refstepcounter{part}%
      \addcontentsline{toc}{part}{Part \thepart.\hspace{1em}#1}%
    \else
      \addcontentsline{toc}{part}{#1}%
    \fi
    {\parindent \z@ \raggedright
     \interlinepenalty \@M
     \normalfont
     \ifnum \c@secnumdepth >\m@ne
     \centering 
     \Large\bfseries \partname\nobreakspace\thepart     
       \nobreak. 
     \fi
     \Large \bfseries { #2}%
     \par}%
    \nobreak
    \vskip 3ex
    \@afterheading}
\def\@spart#1{%
    {\parindent \z@ \raggedright
     \interlinepenalty \@M
     \normalfont
     \huge \bfseries #1\par}%
     \nobreak
     \vskip 3ex
     \@afterheading}
\makeatother

\renewcommand{\thepart}{\Roman{part}}

\makeatletter 
\def\l@subsection{\@tocline{2}{0pt}{3pc}{6pc}{}} 
\makeatother

\begin{document}

\title{The semiregularity theorem for equivariant noncommutative varieties}

\author{Alexander Perry}
\address{Department of Mathematics, University of Michigan, Ann Arbor, MI 48109 \smallskip}
\email{arper@umich.edu}

\begin{abstract} 
We generalize the classical semiregularity theorem of Buchweitz and Flenner to the setting of noncommutative algebraic geometry, with group actions. 
This applies in particular to twisted derived categories, in which case it answers a question of Markman 
and streamlines part of his proof of the Hodge conjecture for abelian fourfolds. 
Along the way, we prove that for many finite group actions on derived categories of varieties, the invariant category is of geometric origin. 
\end{abstract}

\maketitle

\setcounter{tocdepth}{1}
\tableofcontents


\section{Introduction} 

A coherent sheaf $E$ on a variety $X$ is semiregular if the obstruction space $\Ext^2(E,E)$ for its deformations is controlled by Hodge theory, in the sense that the semiregularity map to the Hochschild homology group $\HH_{-2}(X)$ is injective.    
The classical semiregularity theorem of Buchweitz and Flenner \cite{BF} says that a coherent sheaf $E_0$ on the special fiber $X_0$ of a smooth proper family of complex varieties $X \to S$ can be deformed to the generic fiber, provided its Chern character remains of Hodge type along $S$. 

The goal of this paper is to prove a generalization of this theorem in which 
$X$ is allowed to be a ``noncommutative variety'', i.e. a semiorthogonal component of the derived category of a smooth proper family of varieties, and the condition on $E_0$ is relaxed to ``equivariant semiregularity'' in the presence of a suitable group action. 
For simplicity, we first state our result in the motivating case of twisted derived categories, 
without group actions. 

\begin{theorem}
\label{theorem-semiregularity-intro}
    Let $f \colon X \to S$ be a smooth proper family of complex varieties. 
    Let $0 \in S(\bC)$ be a point and let $E_0 \in \Dperf(X_0)$ be a semiregular perfect complex with $\Ext^{<0}(E_0,E_0)=0$. 
    Assume that 
    $B_0 \in \rH^2(X_0, \bQ(1))$ is an algebraic class such that 
    \begin{equation*} 
    w_0 = \exp(B_0) \cdot \ch(E_0) \in \bigoplus_{k \geq 0} \rH^{2k}(X_0, \bQ(k))
    \end{equation*} 
    remains Hodge along~$S$, i.e. lifts to a global section $w$ of the local system $\bigoplus_{k \geq 0} \rR^{2k}f_*\bQ(k)$. 
    Then: 
    \begin{enumerate}
        \item 
        The object $E_0$ deforms as a twisted perfect complex over an \'{e}tale neighborhood of $0$. 
        More precisely, there exists an \'{e}tale morphism $U \to S$, a point $u \in U(\bC)$ mapping to $0 \in S(\bC)$, a Brauer class $\alpha \in \Br(X_{U})$ whose restriction $\alpha_u \in \Br(X_u)$ vanishes, and an $\alpha$-twisted perfect complex $F \in \Dperf(X_U, \alpha)$ whose restriction $F_u \in \Dperf(X_u, \alpha_u)$ satisfies $F_u \simeq E_0$ 
        (under the identification $(X_u, \alpha_u) = (X_0, 0)$).   
        \item The class $w_0$ remains algebraic along~$S$, i.e. for every point $s \in S(\bC)$ the fiber  
        \begin{equation*}
        w_s \in \bigoplus_{k \geq 0} \rH^{2k}(X_s,\bQ(k))
        \end{equation*} 
        of $w$ over $s$ is algebraic. 
    \end{enumerate}
\end{theorem}

This answers a question of Markman \cite[Conjecture~7.3.9]{markman}, who in his beautiful recent work on the Hodge conjecture for abelian fourfolds proved and applied a version of Theorem~\ref{theorem-semiregularity-intro} in the case where $f \colon X \to S$ is a family of abelian varieties and $E_0$ is a coherent sheaf. 
The added flexibility of \emph{complexes} $E_0$ in our theorem is quite useful in practice. 
In particular, this can be used to simplify some of the arguments in \cite{markman}, which was the main motivation for this work; we explain this in \S\ref{section-simplifying-markman} below, after describing some other relevant results.

Theorem~\ref{theorem-semiregularity-intro} is a special case of a much more general semiregularity theorem (Theorem~\ref{theorem-semiregularity-categorical}) proved in the body of the paper, where the role of $X$ is replaced by a smooth proper $S$-linear category $\cC$ of geometric origin, meaning that $\cC$ admits a realization as a semiorthogonal component of a smooth proper $S$-scheme. 
This also yields, for example, a version of Theorem~\ref{theorem-semiregularity-intro} where the given complex $E_0$ may itself be twisted. 
In turn, our general semiregularity theorem 
is based on a smoothness result for the moduli space of objects in $\cC$ at a semiregular point (Theorem~\ref{theorem-semiregular-smooth}), 
which generalizes the smoothness of moduli spaces of objects in CY2 categories proved in \cite{IHC-CY2} (see Remark~\ref{remark-semiregular-smooth-theorem}\eqref{remark-IHC-CY2}). 
The key ingredient in the proof of our smoothness result is Pridham's work on semiregularity \cite{pridham}. 

Finally, we prove an equivariant version of our general semiregularity result, extending the scope of its applications. 
The semiregularity of an object $E \in \cC$ in a $\bC$-linear category $\cC$ is a strong condition which is difficult to check in practice. 
However, in the presence of a suitable finite group of symmetries $G$ acting on $\cC$, it turns out that a more tractable condition suffices. 
Namely, we introduce two notions of equivariant semiregularity for an object $E \in \cC$: 
\begin{itemize} 
\item $E$ is \emph{$G$-semiregular} if it admits a $G$-equivariant structure and the semiregularity map $\Ext^2(E,E) \to \HH_{-2}(\cC)$ is injective when restricted to the \mbox{$G$-invariant} subspace $\Ext^2(E,E)^G$. 
\item $E$ is \emph{weakly $G$-semiregular} if it admits a $G$-equivariant structure such that the corresponding object $\wtilde{E} \in \cC^G$ of the invariant category is semiregular. 
\end{itemize} 
As the names suggest, $G$-semiregularity implies weak $G$-semiregularity (Lemma~\ref{lemma-G-semiregularity}). 
In practice, these conditions may be significantly easier to verify than semiregularity, especially when the symmetry group $G$ is large.  

In Theorem~\ref{theorem-semiregularity-categorical-equivariant} we deduce a $G$-equivariant version of our semiregularity theorem, which says that for suitable group actions, weak $G$-semiregularity has all of the same consequences as ordinary semiregularity. 
The precise general statement is somewhat technical, 
but for abelian varieties we obtain a clean statement: 

\begin{theorem}
\label{theorem-semiregularity-equivariant-intro}
Let $f \colon X \to S$ be a smooth proper family of complex abelian varieties.  
Let \mbox{$0 \in S(\bC)$} be a point. 
Let $G \subset X_0 \times X_0^{\vee}$ be a finite group which acts on $\Dperf(X_0)$ 
(through translations and tensoring by line bundles). 
Let $E_0 \in \Dperf(X_0)$ be a weakly $G$-semiregular perfect complex with $\Ext^{<0}(E_0,E_0)=0$.  
Assume that $B_0 \in \rH^2(X_0, \bQ(1))$ is an algebraic class such that 
\begin{equation*} 
    w_0 = \exp(B_0) \cdot \ch(E_0) \in \bigoplus_{k \geq 0} \rH^{2k}(X_0, \bQ(k))
    \end{equation*} 
remains Hodge along $S$. 
Then: 
\begin{enumerate}
    \item $E_0$ deforms as a twisted perfect complex over an \'{e}tale neighborhood of $0$. 
    \item The class $w_0$ remains algebraic along $S$. 
\end{enumerate} 
\end{theorem}

As a technical ingredient of independent interest, for the proof of Theorem~\ref{theorem-semiregularity-equivariant-intro} we show that many finite group actions on the derived category of a variety can be geometrized, in the sense that they can be realized via automorphisms of a suitable gerbe with finite stabilizers; as a consequence, we show that in many cases the invariant category for the group action is of geometric origin. 
We defer the precise statements of these results to \S\ref{section-geometrization}. 

Our results on equivariant semiregularity were inspired by an argument from \cite{markman}, discussed next. 

\subsection{Relation to Markman's work}
\label{section-simplifying-markman}
The key argument 
in Markman's work \cite{markman} can be summarized as follows. 
He constructs a certain complex $E_0 \in \Dperf(X_0)$ on the special fiber of a family $X/S$ of abelian sixfolds together with an algebraic class $B_0 \in \rH^2(X_0, \bQ(1))$ such that: 
\begin{enumerate}
    \item \label{E0-markman} By construction, there exists an abelian sixfold $Y_0$ (the square of the Jacobian of a genus $3$ curve), a coherent sheaf $F_0$ on $Y_0$ which is $G$-semiregular for a finite subgroup $G \subset Y_0$ (acting by translations), and a derived equivalence $\Phi \colon \Dperf(Y_0) \simeq \Dperf(X_0)$ such that $E_0 = \Phi(F_0)$ (see \cite[Lemma~8.4.1]{markman}). 
    \item \label{w0-markman} The class $w_0 = \exp(B_0) \cdot \ch(E_0)$ remains Hodge along $S$. 
\end{enumerate}
In these terms, the main result of \cite{markman} boils down to proving that $w_0$ remains algebraic along $S$. 

Markman's version of the semiregularity theorem does not directly apply in the above setting, because it requires the object $E_0$ to be a coherent sheaf (whereas $E_0$ above is a perfect complex) and semiregular (not only $G$-semiregular). 
To surmount these issues, Markman devises two clever solutions that occupy \cite[\S9]{markman}. 
For the first issue, he shows that up to a shift, $E_0$ is the derived dual of a coherent sheaf $A_0$ on $X_0$.  
For the second, the idea is to descend $E_0$ (and hence $A_0$) to another abelian sixfold $X'_0$ on which it becomes semiregular; naively, one would like to take $X'_0$ to be the quotient of $X_0$ by $G$, but $G$ only naturally acts on $\Dperf(X_0)$, so this requires some maneuvering. 
The details of these arguments are ad hoc and quite involved. 

Our Theorem~\ref{theorem-semiregularity-equivariant-intro}, on the other hand, applies directly in the situation of \eqref{E0-markman} and~\eqref{w0-markman} above, and thus renders such gymnastics unnecessary. 
Indeed, by a result of Rouquier \cite{rouquier}, the identity component of the group of autoequivalences of the abelian variety $X_0$ is identified with $X_0 \times X_0^{\vee}$. 
Hence the subgroup $G \subset Y_0$ translates via the equivalence $\Phi$ to a subgroup $G \subset X_0 \times X_0^{\vee}$ acting on $\Dperf(X_0)$ for which $E_0$ is $G$-semiregular. 
Moreover $\Ext^{<0}(E_0,E_0) = 0$ because $E_0$ is the image of a coherent sheaf under a derived equivalence. 
Therefore, we may apply Theorem~\ref{theorem-semiregularity-equivariant-intro} to conclude that $w_0$ remains algebraic along $S$. 

In this way, our work puts Markman's into a more robust framework, which streamlines part of his proof. 
We hope that this will be useful for future applications of semiregularity to the Hodge conjecture. 

\subsection{Organization of the paper}
In \S\ref{section-semiregularity} we define (equivariant) semiregularity in our context. 
In \S\ref{section-hodge-theory} we review some preliminaries about the Hodge theory of categories that are needed for the formulation of our results. 
In \S\ref{section-smoothness-moduli} we prove our result on the smoothness of the moduli space of objects at a semiregular point. 
In \S\ref{section-geometrization} we prove our geometrization results for group actions on derived categories. 
Finally, in \S\ref{section-semiregularity-theorems} we combine everything to prove our general noncommutative semiregularity results, as well as Theorems~\ref{theorem-semiregularity-intro} and~\ref{theorem-semiregularity-equivariant-intro} from above for twisted derived categories.  

\subsection{Conventions}
A variety over a field $k$ is a scheme over $k$ which is geometrically integral, separated, and finite type over $k$. 
For a scheme $S$, $\Dqc(S)$ denotes the unbounded derived category of quasi-coherent sheaves, and $\Dperf(S) \subset \Dqc(S)$ denotes the subcategory of perfect complexes. 
All such categories are considered in the enhanced sense, as stable $\infty$-categories. 
All functors are derived. 

We freely use the language of ``categories linear over a base'' described in 
 \cite{NCHPD} and summarized in \cite[\S4]{PI-abelian3folds} and \cite[\S2]{IHC-CY2}. 
Briefly, if $S$ is a base scheme, then an $S$-linear category $\cC$ is a (suitably enhanced) triangulated category with an action of the monoidal category $\Dperf(S)$ of perfect complexes on $S$. 
The key example of such a category is an $S$-linear semiorthogonal component $\cC \subset \Dperf(X)$, where $X$ is an $S$-scheme. 

\subsection{Acknowledgements} 
I am very grateful to James Hotchkiss for sharing several insights that led to the results in \S\ref{section-geometrization} --- in particular for explaining to me Construction~\ref{remark-Aut-bmun-act-DXalpha} and Lemma~\ref{lemma-cX-mod-G-geometric} --- as well as for helpful comments on and corrections to the first version of this paper. 
Special thanks are also due to Eyal Markman for a number of conversations around the topic of the Hodge conjecture for abelian varieties;  
this paper owes its existence to his inspiring work~\cite{markman}. 

This work was partially supported by NSF grants DMS-2052750 and DMS-2143271.


\section{Semiregularity} 
\label{section-semiregularity} 

\subsection{Semiregularity morphisms}
If $\cC$ is a linear category over a base scheme $S$, $\Phi$ is an $S$-linear endofunctor of $\cC$, and $E \in \cC$ is an object, then \cite[Construction 5.10]{PI-abelian3folds} (see also \cite[\S3.2]{IHC-CY2}) defines a \emph{semiregularity morphism} 
\begin{equation*}
\sigma_{E,\Phi} \colon \cHom_S(E, \Phi(E)) \to \cHH_*(\cC/S, \Phi) , 
\end{equation*} 
where $\cHom_S(E, \Phi(E)) \in \Dqc(S)$ is the mapping complex\footnote{In the case where $\cC \subset \Dperf(X)$ is an $S$-linear semiorthogonal component for an $S$-scheme $\pi \colon X \to S$, we have $\cHom_S(E,E) = \pi_* \cHom_X(E,E)$ where $\pi_*$ and $\cHom_X$ are considered as derived functors.} and $\cHH_*(\cC/S, \Phi) \in \Dqc(S)$ is the relative Hochschild homology with coefficients in $\Phi$. 
For any $i \in \bZ$, taking degree $i$ cohomology sheaves gives a morphism which we denote by 
\begin{equation*}
\sigma^i_{E,\Phi} \colon \cExt_S^i(E,\Phi(E)) \to \cHH_{-i}(\cC/S, \Phi). 
\end{equation*}

\begin{remark}
    In \cite{PI-abelian3folds, IHC-CY2}, the morphism $\sigma_{E,\Phi}$ is instead denoted by $\ch_{E,\Phi}$ and called the Chern character, motivated by the fact that
    when $S = \Spec(k)$ is a point and $\Phi = \id_{\cC}$, evaluating $\sigma_{E,\Phi}^0$ on $\id_E$ recovers the $\HH_0(\cC)$-valued Chern character of $E$. 
    Here, we instead emphasize the relation to classical semiregularity maps.
\end{remark}

When $\Phi = - \otimes F$ for an object $F \in \Dperf(S)$, we simplify notation by writing 
\begin{align*}
    \sigma_{E,F} \colon \cHom_S(E, E \otimes F) \to \cHH_*(\cC/S, F) , \\ 
    \sigma^i_{E,F} \colon \cExt_S^i(E,E \otimes F) \to \cHH_{-i}(\cC/S, F), 
\end{align*}
for the above morphisms. 
Similarly, when $\Phi = \id_{\cC}$, we write 
\begin{align*}
    \sigma_{E} \colon \cHom_S(E, E ) \to \cHH_*(\cC/S) , \\ 
    \sigma^i_{E} \colon \cExt_S^i(E,E) \to \cHH_{-i}(\cC/S). 
\end{align*}

\begin{remark}
\label{remark-HH-F}
    For $F \in \Dperf(S)$, it follows directly from the definition of Hochschild homology \cite[Definition 3.1]{IHC-CY2} that there is an equivalence  
    \begin{equation*} 
    \cHH_*(\cC/S, F) \cong \cHH_*(\cC/S) \otimes F. 
    \end{equation*} 
    In other words, the theory with coefficients in an object $F \in \Dqc(S)$ reduces to the case of trivial coefficients. 
\end{remark}

\begin{remark}
\label{remark-semiregularity-map-functorial}
    The semiregularity morphism is suitably functorial. 
    This is simplest to state in the case where the coefficients $\Phi = \id_{\cC}$ are trivial. 
    Then if $\Psi \colon \cC \to \cD$ is a functor between $S$-linear categories, for any object there is an induced commutative diagram 
    \begin{equation*}
        \xymatrix{
        \cHom_S(\Psi(E), \Psi(E)) \ar[rr]^{ \sigma_{\Psi(E)} }  && \cHH_*(\cD/S)  \\ 
        \cHom_S(E, E) \ar[rr]^{\sigma_{E}} \ar[u] && \cHH_*(\cC/S) \ar[u]_{\Psi_*}
        }
    \end{equation*}
    where the left vertical map is given by functoriality of $\Psi$ and $\Psi_*$ is the induced map on Hochschild homology (see \cite[\S3.1]{IHC-CY2}). 
\end{remark}

\subsection{Semiregularity of objects} 
We will mostly be interested in the semiregularity morphism when $S = \Spec(k)$ is the spectrum of a field and $\Phi = \id_{\cC}$ is the identity. In this case, $\cExt_k^i(E,E)$ and $\cHH_{-i}(\cC/k)$ correspond to the usual $k$-vector spaces $\Ext_k^i(E,E)$ and $\HH_{-i}(\cC)$. 

\begin{definition}
\label{definition-semiregular}
Let $\cC$ be a $k$-linear category for a field $k$. 
An object $E \in \cC$ is \emph{semiregular} if the map 
\begin{equation*} 
\sigma_{E}^2 \colon \Ext_k^2(E,E) \to \HH_{-2}(\cC)
\end{equation*} 
is injective. 
\end{definition} 

\begin{remark}
When $\cC = \Dperf(X)$ for a smooth proper complex variety $X$, then $\sigma_E^2$ can be identified (after applying an HKR isomorphism) with the semiregularity map of Buchweitz and Flenner \cite{BF}. 
\end{remark} 

Next we discuss the equivariant case. 
We follow \cite[\S3]{bayer-perry} for our conventions on group actions on linear categories and their invariants. 

\begin{definition}
Let $\cC$ be a $k$-linear category, where $k$ is a field. Assume $\cC$ is equipped with an action by a finite group $G$.
\begin{itemize}
\item An object $E \in \cC$ is \emph{$G$-semiregular} if:
\begin{enumerate}
    \item $E$ admits a $G$-equivariant structure, i.e. there exists an object $\wtilde{E} \in \cC^G$ whose image under the forgetful functor $\Forg \colon \cC^G \to \cC$ is $E$. 
    \item The restriction of the map $\sigma_{E}^2 \colon \Ext_k^2(E,E) \to \HH_{-2}(\cC)$ to the $G$-invariant subspace $\Ext^2_k(E,E)^G$  is injective. 
\end{enumerate}
\item An object $E \in \cC$ is \emph{weakly $G$-semiregular} if there exists a semiregular object $\wtilde{E} \in \cC^G$ such that $\Forg(\wtilde{E}) \simeq E$. 
\end{itemize}
\end{definition}

\begin{lemma}
\label{lemma-G-semiregularity}
Let $\cC$ be a $k$-linear category, where $k$ is a field. Assume $\cC$ is equipped with an action by a finite group $G$ whose order is invertible in $k$. 
Let $E \in \cC$ be a $G$-semiregular object. 
Then $E$ is weakly $G$-semiregular. 
More precisely, any object $\wtilde{E} \in \cC^G$ with $\Forg(\wtilde{E}) \simeq E$ is semiregular. 
\end{lemma}

\begin{proof}
    In view of Remark~\ref{remark-semiregularity-map-functorial}, we have a commutative diagram 
    \begin{equation*}
        \xymatrix{
        \Ext^2_k(E,E) \ar[rr]^{ \sigma^2_{E} }  && \HH_{-2}(\cC)  \\ 
        \Ext_k^2(\wtilde{E}, \wtilde{E}) \ar[rr]^{\sigma^2_{\wtilde{E}}} \ar[u] && \HH_{-2}(\cC^G) .  \ar[u]  
        }
    \end{equation*}
    In $\cC^G$, the mapping object is computed by 
    \begin{equation*}
        \cHom_k(\wtilde{E}, \wtilde{E}) \simeq \cHom_k(E,E)^G.
    \end{equation*}
    By our assumption on the order of $G$, it follows that 
    \begin{equation*}    
    \Ext_k^2(\wtilde{E}, \wtilde{E}) \cong \Ext^2_k(E,E)^G.
    \end{equation*} 
Under this isomorphism, the composition of the left vertical arrow with $\sigma^2_E$ in the above diagram is identified with the restriction of $\sigma^2_E$ to $\Ext^2_k(E,E)^G$. The claim follows. 
\end{proof}


\section{Hodge theory of categories} 
\label{section-hodge-theory}

Let $S$ be a complex variety. 
To formulate our semiregularity theorem, 
we will need to consider an $S$-linear category $\cC$ which supports a reasonable version of Hodge theory. 
Such a theory was developed in \cite{IHC-CY2} when $\cC$ is a \emph{smooth proper $S$-linear category of geometric origin},
meaning that $\cC$ admits an embedding as an $S$-linear semiorthogonal component into $\Dperf(X)$ for a smooth proper $S$-scheme $X$.\footnote{Conjecturally, the theory also extends to the case of abstract smooth proper $S$-linear categories.} 

Briefly, for a smooth proper $S$-linear category $\cC$ of geometric origin, there is a local system $\Ktop[0](\cC/S)$ underlying a weight $0$ variation of Hodge structures on $S$, which is compatible with restriction to fibers and the classical Hodge theory of varieties. 
In this section, we review some relevant results and examples of this theory. 

\subsection{General formalism}
As above, let $\cC$ be a smooth proper $S$-linear category of geometric origin, where $S$ is a complex variety. 

\subsubsection*{Absolute case} 
When $S = \Spec(\bC)$ is a point, we write $\Ktop[0](\cC)$ for $\Ktop[0](\cC/\Spec(\bC))$.
There is a canonical map $\rK_0(\cC) \to \Ktop[0](\cC)$ from the Grothendieck group of $\cC$ to the topological K-theory, and classes in the image are called \emph{algebraic}.
This map factors through the subgroup $\Hdg(\cC, \bZ) \subset \Ktop[0](\cC)$ of integral Hodge classes \cite[Lemma 5.10]{IHC-CY2}. 
In these terms, the \emph{integral Hodge conjecture} for $\cC$ states that $\rK_0(\cC) \to \Hdg(\cC, \bZ)$ is surjective, while the \emph{Hodge conjecture} states that this is so after tensoring with $\bQ$. 

Moreover, the construction $\Ktop[0](\cC)$ is additive under semiorthogonal decompositions, so that if $\cC \subset \Dperf(X)$ is a semiorthogonal component of a smooth proper complex variety $X$, then $\Ktop[0](\cC)$ is naturally a summand of $\Ktop[0](\Dperf(X))$. 
In this sense, the theory essentially reduces to the geometric case discussed next.  

\subsubsection*{Absolute geometric case}
When $S = \Spec(\bC)$ is a point and $\cC = \Dperf(X)$ for a smooth proper variety $X$, then $\Ktop[0](\cC)$ identifies with the complex topological $\rK$-theory $\Ktop[0](X)$ of (the analytification of) $X$. 
Moreover, the Chern character gives an isomorphism 
\begin{equation*}
\Ktop[0](\Dperf(X)) \otimes \bQ \cong \rH^{\ev}(X, \bQ) \coloneqq \bigoplus_{k \geq 0} \rH^{2k}(X, \bQ)(k) 
\end{equation*} 
of weight $0$ rational Hodge structures \cite[Proposition~5.4]{IHC-CY2}. 
Similarly, after tensoring with $\bQ$, the map $\rK_0(\Dperf(X)) \to \Hdg(\Dperf(X), \bZ)$ identifies (via the Chern character) with the cycle class map $\CH^*(X) \otimes \bQ \to \Hdg^*(X, \bQ)$ to the group of rational Hodge classes (summed up over all degrees). In this way, the usual Hodge conjecture for $X$ in all degrees is equivalent to the Hodge conjecture for $\Dperf(X)$ \cite[Lemma 5.12]{IHC-CY2}. 

\subsubsection*{Relative case} 
For general $S$, the fiber of the variation of Hodge structures $\Ktop[0](\cC/S)$ over $s \in S(\bC)$ is the Hodge structure $\Ktop[0](\cC_s)$ on the $\bC$-linear fiber category $\cC_s$. 
Moreover, the theory is additive under semiorthogonal decompositions, so that if $\cC \subset \Dperf(X)$ is a semiorthogonal component of a smooth proper $S$-scheme $X$, then $\Ktop(\cC/S)$ is a summand of $\Ktop[0](\Dperf(X)/S)$. 
In this sense, the theory essentially reduces to the relative geometric case discussed next. 

\subsubsection*{Relative geometric case} 
When $S$ is general and $\cC = \Dperf(X)$ for a smooth proper morphism $f \colon X \to S$, 
then $\Ktop[0](\cC/S)$ identifies with the local system $\Ktop[0](X/S)$ of topological $K$-theories of the fibers of $X \to S$. 
Moreover, generalizing the absolute case above, there is an isomorphism 
\begin{equation*}
    \Ktop[0](\Dperf(X)/S) \otimes \bQ \cong \rR^{\ev}f_*\bQ \coloneqq 
    \bigoplus_{k \geq 0} \rR^{2k}f_* \bQ(k)
\end{equation*}
of weight $0$ variations of rational Hodge structures on $S$ \cite[Proposition 5.7]{IHC-CY2}.

\subsection{Twisted derived categories}
Given a Brauer class $\alpha \in \Br(X) \coloneqq \rH^2_{\et}(X, \bG_m)_{\tors}$ on a scheme $X$, we may consider the category $\Dperf(X, \alpha)$ of $\alpha$-twisted perfect complexes. 
There are various models for this category, each depending on an auxiliary choice, but all of these models result in equivalent theories which are independent of the choices involved (see~\cite{lieblich-moduli-twisted}). 
For us, it is convenient to choose a $\bmu_n$-gerbe $\cX \to X$ whose class maps to $\alpha$ under $\rH^2_{\et}(X, \bmu_n) \to \rH^2_{\et}(X, \bG_m)$, and then define $\Dperf(X, \alpha)$ as the category of $1$-twisted perfect complexes on $\cX$. 

Now assume that $f \colon X \to S$ is a smooth proper family of complex varieties, 
and that $\alpha$ can be represented by an Azumaya algebra\footnote{By \cite{dJ-gabber}, when $X$ admits an ample line bundle then every Brauer class admits such a representation.}. Then $\Dperf(X, \alpha)$ 
is a smooth proper $S$-linear category of geometric origin by \cite{bernardara-BS, bergh-BS} (see~\cite[Example 4.11]{PI-abelian3folds}), so we may consider its Hodge theory in the sense discussed above. 
Below we explain some of Hotchkiss's work \cite{hotchkiss-pi} on the resulting Hodge structure. 

First we consider the absolute case $S = \Spec(\bC)$. 
We also assume that $\alpha$ is \emph{topologically trivial}, or equivalently that there exists a class $B \in \rH^2(X, \bQ(1))$ mapping to $\alpha$ under 
\begin{equation*}
    \exp \colon \rH^2(X, \bQ(1)) \to \rH^2(X^{\an}, \bG_m)_{\tors} = \rH^2_{\et}(X, \bG_m)_{\tors}. 
\end{equation*}
The \emph{$B$-twisted Mukai Hodge structure} $\Ktop[0](X)^{B}$ is 
the unique weight $0$ integral Hodge structure on $\Ktop[0](X)$ such that the homomorphism 
\begin{equation*}
    \Ktop[0](X)^{B} \otimes \bQ \to  \rH^{\ev}(X,\bQ), \quad v \mapsto \exp(B) \cdot \ch(v)
\end{equation*}
is an isomorphism of $\bQ$-Hodge structures, where here $\exp(B) = 1 + B + \frac{B^2}{2!} + \frac{B^3}{3!} + \cdots$. 
Then by \cite[Theorem 4.10]{hotchkiss-pi} there is an isomorphism 
\begin{equation}
\label{varphi-KtopB}
    \varphi \colon \Ktop[0](\Dperf(X, \alpha)) \xrightarrow{\sim} \Ktop[0](X)^B
\end{equation}
of weight $0$ integral Hodge structures; in particular, there is an isomorphism 
\begin{equation*}
\Ktop[0](\Dperf(X, \alpha)) \otimes \bQ \xrightarrow{\sim} \rH^{\ev}(X,\bQ)
\end{equation*}
of $\bQ$-Hodge structures.

\begin{remark}
\label{remark-trivial-brauer-class}
    Suppose that the class $B \in \rH^2(X, \bQ(1))$ is algebraic, so that the Brauer class $\alpha = 0$ vanishes. Then there is an equivalence 
    $\Dperf(X, \alpha) \simeq \Dperf(X)$ under which the isomorphism~\eqref{varphi-KtopB} is identified with the canonical isomorphism $\Ktop(\Dperf(X)) \cong \Ktop[0](X)$ (see \cite[Remark~18.6]{PI-abelian3folds}). 
\end{remark}

Now we return to the relative situation where $f \colon X \to S$ is a smooth proper family of complex varieties. 
We assume that $\alpha$ is relatively topologically trivial, in the sense that its restriction $\alpha_s \in \Br(X_s)$ is topologically trivial for each $s \in S(\bC)$. 
Recall that in order to define $\Dperf(X,\alpha)$, we also choose a $\bmu_n$-gerbe $\cX \to X$ whose class $\theta \in \rH^2_{\et}(X, \bmu_n)$ maps to $\alpha \in \rH^2_{\et}(X, \bG_m)^{\tors}$; this choice is important for the formulation of the result below. 
In terms of $\theta$, the condition that $\alpha_s$ is topologically trivial amounts to the condition that the restriction $\theta_s \in \rH^2_{\et}(X_s, \bmu_n)$ lies in the image of the map 
\begin{equation}
\label{bs-field}
    \exp(-/n) \colon \rH^2(X_s, \bZ(1)) \to \rH^2(X_s^{\an}, \bmu_n) = \rH^2_{\et}(X_s, \bmu_n)
\end{equation}

\begin{theorem}[{\cite[\S4.4]{hotchkiss-pi}}]
\label{theorem-mukai-variation}
In the above situation, there is an isomorphism 
\begin{equation*}
    \gamma \colon \Ktop[0](\Dperf(X,\alpha)/S) \otimes \bQ \xrightarrow{\sim} \rR^{\ev} f_*\bQ
\end{equation*}
of $\bQ$-variations of Hodge structure\footnote{As in the absolute case, there is an integral statement, but its formulation is slightly more elaborate and we will not need it.} with the following properties: 
\begin{enumerate}
    \item The formation of $\gamma$ is compatible with base change on $S$. 
    \item \label{commutative-diagram-gammas} 
    Let $s \in S(\bC)$, 
    let $b_s \in \rH^2(X_s, \bZ(1))$ be an element with image $\theta_s$ under the map~\eqref{bs-field}, 
    and let $B_s = b_s/n \in \rH^2(X_s, \bQ(1))$. 
    Then there is a commutative diagram 
        \begin{equation}
             \begin{tikzcd}
                     \Ktop[0](\Dperf(X_s, \alpha_s)) \otimes \bQ \ar{d}[swap] {\varphi_s} \ar[r, "\gamma_s"] & \rH^{\ev}(X_s, \bQ) \\
                     \Ktop[0](X_s)^{B_s} \otimes \bQ \ar[ur, "\exp(B_s) \cdot \ch(-)"']
             \end{tikzcd}
        \end{equation}
        where $\gamma_s$ is the fiber of $\gamma$ over $s$ and $\varphi_s$ is the ($\bQ$-extension of the) isomorphism~\eqref{varphi-KtopB} for $(X_s, \alpha_s)$.  
    \item \label{mukai-variation-section} 
    In the situation of~\eqref{commutative-diagram-gammas}, 
    a class $v_s \in \Ktop(\Dperf(X_s, \alpha_s)) \otimes \bQ$ lifts to a global section of the variation $\Ktop[0](\Dperf(X,\alpha)/S) \otimes \bQ$ if and only if $\exp(B_s) \cdot \ch(\varphi_s(v_s))$ lifts to a global section of $\rR^{\ev}f_*\bQ$. 
\end{enumerate}
\end{theorem} 

\subsection{Invariant categories} 
\label{hodge-theory-invariant-categories}
Let $\cC$ is a smooth proper category of geometric origin over a complex variety $S$, and assume that $\cC$ is equipped with an action by a finite group $G$. 
To our knowledge, there is no known simple formula for $\Ktop[0](\cC^G/S)$ in terms of $\Ktop[0](\cC/S)$. 
On the other hand, the situation is simpler if we focus on the invariant parts and work rationally. 
Here, we use that there is a natural \emph{residual action} of the character group $G^{\vee} \coloneqq \Hom(G, \bC^{\times})$ on $\cC^G$ by ``rescaling $G$-linearizations'' (see for instance \cite{elagin} or the forthcoming \cite{NC-abelian-surfaces}). 
Hence $G^{\vee}$ also naturally acts on $\Ktop[0](\cC^G/S)$. 

\begin{lemma}
\label{lemma-Ktop-G}
    Let $\cC$ be a smooth proper $S$-linear category of geometric origin, where $S$ is a complex variety. 
    Let $G$ be a finite abelian group which acts on $\cC$ as an $S$-linear category in such a way that $\cC^G$ is smooth and proper of geometric origin over $S$.  
    Then the forgetful functor $\Forg \colon \cC^G \to \cC$ induces an isomorphism  
    \begin{equation*}
        \Ktop[0](\cC^G/S)^{G^\vee} \otimes \bQ \cong \Ktop[0](\cC/S)^G \otimes \bQ
    \end{equation*}
    of $\bQ$-variations of Hodge structures.
\end{lemma}

\begin{remark}
    The condition that $\cC^G$ is of geometric origin is subtle even when $\cC = \Dperf(X)$ is the derived category of a smooth proper $S$-scheme, but in \S\ref{section-geometrization} we prove that it holds for many group actions in this situation.  
\end{remark}

\begin{proof}
    One checks that the forgetful and inflation functors $\Forg \colon \cC^G \to \cC$ and $\Inf \colon \cC \to \cC^G$ induce maps
    \begin{equation*}
        \Ktop[0](\cC^G/S)^{G^{\vee}} \to \Ktop[0](\cC/S)^G \quad \text{and} \quad 
        \Ktop[0](\cC/S)^G \to \Ktop[0](\cC^G/S)^{G^{\vee}}
    \end{equation*}
    whose compositions in either direction are multiplication by $|G|$. 
    We refer to the forthcoming work \cite{NC-abelian-surfaces} for such arguments in a more general setting. 
\end{proof}

Let $\cC$ be a smooth proper $S$-linear category, which is connected over $S$ in the sense of \cite[Definition~5.18]{PI-abelian3folds}; 
for instance, this holds when $\cC = \Dperf(X)$ where $f \colon X \to S$ is a smooth proper morphism of schemes with geometrically integral fibers. 
Then by \cite[Proposition~8.2]{PI-abelian3folds}, there is a group algebraic stack $\cAut_{\cC/S}$ locally of finite type over $S$, which is a $\bG_m$-gerbe over a group algebraic space $\Aut_{\cC/S}$; when $S = \Spec(k)$ is the spectrum of a field, we often omit it from the notation.  
Below, we consider the situation where there exists a point $0 \in S(\bC)$ such that $G$ acts on the fiber $\cC_0$ through the identity component $\Aut^0_{\cC_0}$ of $\Aut_{\cC_0}$, i.e. such that the associated homomorphism 
$G \to \Aut_{\cC_0}(\bC)$ factors through $\Aut^0_{\cC_0}(\bC)$. 

\begin{lemma}
\label{lemma-G-action-Ktop-trivial}
    Let $\cC$ be a smooth proper $S$-linear category of geometric origin, where $S$ is a complex variety. 
    Let $G$ be a finite group which acts on $\cC$ as an $S$-linear category. 
    Assume that there exists a point $0 \in S(\bC)$ such that $\cC_0$ is a connected $\kappa(0)$-linear category and the induced action of $G$ on $\cC_0$ is through the identity component $\Aut^0_{\cC_0}$. 
    Then $G$ acts trivially on the local system $\Ktop[0](\cC/S)$, so that  
    \begin{equation*}
        \Ktop[0](\cC/S)^G = \Ktop[0](\cC/S). 
    \end{equation*}
\end{lemma}

\begin{proof}
    Since $\Ktop[0](\cC/S)$ is a local system, it suffices to show that $G$ acts trivially on the fiber $\Ktop[0](\cC_0)$. 
    This holds because by assumption $G$ acts through the connected group $\Aut^0(\cC_0)$. 
\end{proof}


\section{Smoothness of moduli spaces} 
\label{section-smoothness-moduli}

\subsection{Moduli spaces} 
Let $\cC$ be a smooth proper category of geometric origin over a base scheme $S$. 
Then there is an algebraic stack $\cM(\cC/S)$, locally of finite presentation over $S$, which parameterizes the \emph{gluable} objects in $\cC$, i.e. the objects with vanishing negative self-Ext groups. 
Explicitly, the modui functor is given by 
\begin{alignat}{3}
\nonumber \cM(\cC/S) \colon  &  (\Sch/S)^{\op} ~~ &  \to & ~~ \Grpd & \\ 
\nonumber & \qquad T & \mapsto 
& \set{E \in \cC_T ~ | ~  \Ext^{<0}_{\kappa(t)}(E_t, E_t) = 0 \text{ for all } t \in T} & ,  
\end{alignat} 
where $E_t$ denotes restriction of $E$ to the fiber category $\cC_t$. 
The algebraicity of $\cM(\cC/S)$ can be deduced from Lieblich's work \cite{lieblich}, and is established more generally for smooth proper categories not necessarily of geometric origin by To\"{e}n and Vaqui\'{e} \cite{toen-moduli} (see the discussions in \cite[\S9]{stability-families} and \cite[\S7.1]{IHC-CY2}). 

Now assume that $S$ is a complex variety, so that we have at our disposal the Hodge theory for $\cC$ discussed in \S\ref{section-hodge-theory}. 
Then any object $E \in \cC$ gives rise to a section $v_E$ of the local system $\Ktop[0](\cC/S)$. 
If we are given a section $v$ of the local system $\Ktop[0](\cC/S)$, we can thus consider the subfunctor $\cM(\cC/S, v)$ of $\cM(\cC/S)$ parameterizing objects of class $v$. 
Explicitly, the moduli functor is given by 
\begin{alignat}{3}
\nonumber \cM(\cC/S, v) \colon  &  (\Sch/S)^{\op} ~~ &  \to & ~~ \Grpd & \\ 
\nonumber & \qquad T & \mapsto 
& \set{E \in \cC_T ~ | ~  v_{E} = v_T \text{ and } \Ext^{<0}_{\kappa(t)}(E_t, E_t) = 0 \text{ for all } t \in T} & ,  
\end{alignat} 
where $v_T$ denotes the pullback of $v$ to a section of $\Ktop[0](\cC_T/T)$. 
Then $\cM(\cC/S, v) \subset \cM(\cC/S)$ is an open substack, and hence in particular also algebraic and locally of finite type over $S$ \cite[Lemma 7.3]{IHC-CY2}. 

\subsection{Smoothness} 
The following smoothness result for the above moduli space is one of the main technical ingredients in this paper. 

\begin{theorem}
\label{theorem-semiregular-smooth} 
Let $\cC$ be a smooth proper $S$-linear category of geometric origin, where $S$ is a complex variety. 
Let $v$ be a section of $\Ktop[0](\cC/S)$. 
Let $0 \in S(\bC)$ be a point and let $E_0 \in \cC_0$ be a gluable semiregular object whose class is equal to the fiber $v_0$ of $v$ at $0$. 
Then the morphism $\cM(\cC/S, v) \to S$ is smooth at $E_0$. 
\end{theorem} 

\begin{remark}
\label{remark-semiregular-smooth-theorem}
Before turning to the proof, we make some comments on the theorem. 
\begin{enumerate}
\item \label{remark-IHC-CY2} 
Theorem~\ref{theorem-semiregular-smooth} generalizes the smoothness result \cite[Theorem 1.4]{IHC-CY2} for moduli of simple gluable objects in a CY2 category (which in turn generalizes Mukai's classical smoothness result for moduli of simple sheaves on a CY2 surface). 
Indeed, if $E_0 \in \cC$ is a simple gluable object in a CY2 category, then $E_0$ is automatically semiregular by \cite[Lemma 7.4]{IHC-CY2}. 

\item \label{semiregular-smooth-theorem-rational} There is a slight variant of Theorem~\ref{theorem-semiregular-smooth} --- which holds by the same proof --- where we consider a section of $\Ktop[0](\cC/S) \otimes \bQ$ instead of $\Ktop[0](\cC/S)$. 

\item 
By general theory, the fibers $v_s \in \Ktop[0](\cC_s)$ of $v$ are automatically Hodge for all $s \in S(\bC)$, even only assuming that $v_0$ is a Hodge class (as opposed to the class of a gluable semiregular object) \cite[Lemma 5.22]{IHC-CY2}. 
\end{enumerate}
\end{remark}

\begin{proof}[Proof of Theorem~\ref{theorem-semiregular-smooth}]
As the morphism $\cM(\cC/S,v) \to S$ is locally of finite type, 
it suffices to prove that it is formally smooth at the $\bC$-point $E_0$. 
In other words, let 
\begin{equation}
\label{sqz}
    0 \to I \to A' \to A \to 0
\end{equation} 
be a square-zero extension of Artinian local $\bC$-algebras with residue field $\bC$, let $\Spec(A') \to S$ be a morphism taking the closed point to $0 \in S$, and let $E \in \cC_A$ be a gluable object (where $\cC_A$ is the base change of $\cC$ along $\Spec(A) \to \Spec(A') \to S$) whose restriction to $\cC_0$ is isomorphic to $E_0$. 
Then we must show that $E$ lifts to an object $E' \in \cC_{A'}$. 
Such a lift exists if and only if the corresponding obstruction class 
\begin{equation*}
    \ob(E) \in \Ext^2_A(E, E \otimes I) 
\end{equation*}
vanishes (see \cite[Lemma 4.8]{IHC-CY2}).

Since the class of $E$ remains of Hodge type along $S$ by assumption, 
it follows from Pridham's work \cite{pridham} that the semiregularity map 
\begin{equation*}
    \sigma_{E,I}^2 \colon \Ext^2_A(E,E \otimes I) \to \HH_{-2}(\cC_A/A, I)
\end{equation*}
kills the obstruction class $\ob(E)$. 
More precisely, if $X/S$ is a smooth proper $S$-scheme for which there exists an embedding $\cC \hookrightarrow \Dperf(X)$ as a semiorthogonal component, 
then the description of $\Ktop[0](\cC/S)$ as a summand of $\Ktop[0](X/S)$ together with \cite[\S2.5]{pridham} shows that $\sigma_{E,I}^2$ kills $\ob(E)$. 

By a standard argument (see \cite[\href{https://stacks.math.columbia.edu/tag/06GE}{Tag 06GE}]{stacks-project}), 
to prove the desired lift of $E$ to $\cC_{A'}$ exists, we may reduce to the case where~\eqref{sqz} is a small extension, i.e. where $I$ is annihilated by the maximal ideal $\fm_{A'}$ and $I \cong \bC$ as a vector space over $\bC = A'/\fm_{A'}$. 
In this case, note that if $i \colon \Spec(\bC) \to \Spec(A)$ is the closed point, then $E \otimes I \cong i_*i^*E \cong i_*(E_0)$, and thus adjunction gives an equivalence of mapping complexes  
\begin{equation*}
    \cHom_A(E, E \otimes I) \cong i_*\cHom_\bC(E_0, E_0).
\end{equation*}
Similarly, we have equivalences 
\begin{equation*}
\cHH_*(\cC_A/A, I)  \simeq \cHH_*(\cC_A/A) \otimes I \simeq i_* i^*\cHH_*(\cC_A/A)  \simeq i_* \cHH_*(\cC_0/\bC), 
\end{equation*}
where the first holds by Remark~\ref{remark-HH-F} and the last by base change for Hochschild homology \cite[Lemma 3.4]{IHC-CY2}. 
Under the above identifications, the semiregularity morphism 
\begin{equation*}
    \sigma_{E,I} \colon \cHom_A(E, E \otimes I) \to \cHH_*(\cC_A/A, I) 
\end{equation*}
identifies with (the pushforward along $i$) of the semiregularity morphism 
\begin{equation*}
    \sigma_{E_0} \colon \cHom_{\bC}(E_0, E_0) \to \cHH_*(\cC_0/\bC). 
\end{equation*}
In particular, since $E_0$ is semiregular, we find that $\sigma_{E,I}^2$ is injective. 
Since we saw above that $\sigma_{E,I}^2$ also kills $\ob(E)$, we find that $\ob(E) = 0$ and the desired lift of $E$ to $\cC_{A'}$ exists. 
\end{proof}


\section{Geometrizing group actions on derived categories} 
\label{section-geometrization} 

Let $\cC$ be a smooth proper $S$-linear category of geometric origin, where $S$ is a variety over a field of characteristic $0$, and assume that $\cC$ is equipped with an action by a finite group~$G$. 
In this setting, the invariant category $\cC^G$ is a smooth proper $S$-linear category \cite[Proposition~3.15]{bayer-perry}, but it is unknown whether $\cC^G$ is necessarily of geometric origin, even if $\cC = \Dperf(X)$. 
Our goal here is to explain that the situation is better when $G$ acts through the identity component of the group of autoequivalences, focusing on the case when 
$\cC$ is the (twisted) derived category of a variety. 

The key observation, due to Olsson \cite{olsson} and explained to us by James Hotchkiss, is that many autoequivalences of $\Dperf(X)$ can be lifted to \emph{automorphisms} of a suitable gerbe. 
When a group action on $\Dperf(X)$ lifts to an action on a gerbe in this way, the invariant category is a semiorthogonal component of the quotient stack of the gerbe; in good cases, this implies that the invariant category is of geometric origin. 

\subsection{Automorphisms of gerbes and their action on derived categories}

\begin{notation}
Let $X \to S$ be a smooth proper morphism of schemes with geometrically integral fibers. 
We write: 
\begin{itemize}
    \item $\cPic_{X/S}$ for the Picard stack of $X$ over $S$, which is a $\bG_m$-gerbe over the Picard space $\Pic_{X/S}$.
    \item $\Aut_{X/S}$ for the automorphism space of $X$ over $S$.
\end{itemize} 
Then $\cPic_{X/S}$ is a group algebraic stack locally of finite type over $S$, and $\Pic_{X/S}$ and $\Aut_{X/S}$ are group algebraic spaces locally of finite type over $S$. 
\end{notation}

\begin{definition}
    If $G \to \Spec(k)$ is a group algebraic space locally of finite type over a field, we denote by $G^{0} \subset G$ the \emph{identity component} of $G$, i.e. the open subgroup given by the connected component of the identity. 
If $G \to S$ is a group algebraic space locally of finite type over a scheme $S$, we consider the union of the the identity components $G_s^0 \subset G_s$, $s \in S$, of all fibers; when this is an open subset $G^0 \subset G$, then it is a subgroup called the \emph{identity component} of $G \to S$, and we say ``the identity component exists'' to signify this situation.
\end{definition}

While the identity component does not exist in general, it often does when $G \to S$ is a group of automorphisms or autoequivalences associated to a family of Calabi--Yau varieties or categories; see e.g. \cite[Lemma~8.11 and~8.14]{PI-abelian3folds}. We will also describe an instance of this phenomenon in Theorem~\ref{theorem-Aut-mun-gerbe}\eqref{identity-components} below. 

\begin{theorem}[{\cite{olsson}}]
\label{theorem-Aut-mun-gerbe}
    Let $f \colon X \to S$ be a smooth proper morphism of schemes with geometrically integral fibers. 
    Let $\cX \to X$ be a $\bmu_n$-gerbe, and let $\cY \to X$ be its pushout to a $\bG_m$-gerbe. 
    Let $\cAut_{\cX/S}$ and $\cAut_{\cY/S}$ be the stacks of automorphisms of $\cX$ and $\cY$ over $S$ which induce the identity on stabilizer group schemes $\bmu_n$ and~$\bG_m$. 
    \begin{enumerate}
        \item $\cAut_{\cX/S}$ is a group algebraic stack locally of finite type over $S$, which is a $\bmu_n$-gerbe over a group algebraic space $\Aut_{\cX/S}$. 
        \item \label{cAutY-exists} $\cAut_{\cY/S}$ is a group algebraic stack locally of finite type over $S$, which is a $\bG_m$-gerbe over a group algebraic space $\Aut_{\cY/S}$. 
        \item \label{Aut-ses}
        The natural morphism $\cAut_{\cX/S} \to \cAut_{\cY/S}$ of group algebraic stacks induces a commutative diagram of exact sequences of group algebraic spaces over $S$: 
        \begin{equation}
        \label{Pic-Aut-Aut}
        \begin{tikzcd}
        1 \ar{r} & \Pic_{X/S}[n] \ar{r} \ar{d} & \Aut_{\cX/S} \ar{r} \ar{d} & \Aut_{X/S} \ar{d} \\ 
        1 \ar{r} & \Pic_{X/S} \ar{r} & \Aut_{\cY/S} \ar{r} & \Aut_{X/S} . 
        \end{tikzcd}
        \end{equation}
        \item \label{Aut-0-surjective}
        If $n$ is invertible on $S$ and the identity components $\Aut_{\cX/S}^0$, $\cAut_{\cY/S}^0$, and $\Aut_{X/S}^0$ exist, then the maps $\Aut_{\cX/S}^0 \to \Aut_{X/S}^0$ and $\Aut_{\cY/S}^0 \to \Aut_{X/S}^0$ on identity components are surjective. 
        \item \label{identity-components} 
        If $\omega_{X/S} = f^*L$ for a line bundle $L$ on $S$ and $S$ is a locally noetherian $\bQ$-scheme, 
        then the identity components $\Aut_{X/S}^0$, $\Aut_{\cX/S}^0$, and $\Aut_{\cY/S}^0$ exist and are finite type over~$S$; moreover, if $S$ is reduced, then they are smooth and proper over $S$. 
    \end{enumerate}
\end{theorem}

\begin{proof}
    For the $\bG_m$-gerbe $\cY$ and $S = \Spec(k)$ a point, the claims in \eqref{cAutY-exists}-\eqref{Aut-0-surjective} hold by \cite[Theorem 1.1]{olsson}. 
    The case of general $S$ holds by the same argument (or by reducing to the case of a field by passing to fibers), as do the corresponding statements for the $\bmu_n$-gerbe $\cX \to X$. 
    Finally, \eqref{identity-components} holds by the argument in the proof of \cite[Lemma 8.14]{PI-abelian3folds}.  
\end{proof}

\begin{remark}
    In the situation of Theorem~\ref{theorem-Aut-mun-gerbe}, there are also similar results for the stack of autoequivalences of $\Dperf(X, \alpha)$, where $\alpha \in \Br(X)$ is the class associated to $\cX \to X$. 
    More precisely, there is a group algebraic stack $\cAut_{\Dperf(X,\alpha)/S}$ locally of finite type over $S$ parameterizing such autoequivalences, which is a $\bG_m$-gerbe over an algebraic space $\Aut_{\Dperf(X,\alpha)/S}$ --- see \cite[Proposition~8.2]{PI-abelian3folds} for a more general result, which holds for smooth proper connected categories over $S$. 
    Moreover, by \cite[Lemma~8.14]{PI-abelian3folds} already invoked in the proof above, under the assumptions of Theorem~\ref{theorem-Aut-mun-gerbe}\eqref{identity-components} the identity component $\Aut^0_{\Dperf(X,\alpha)/S}$ exists and is smooth and proper over $S$. 
\end{remark}

\begin{notation}
    When the base $S = \Spec(k)$ is the spectrum of a field which is clear from context, we often omit $S$ in the notation for the automorphism stacks and spaces discussed above. 
    For instance, in this case, we write simply $\Aut_{\cX}$ for $\Aut_{\cX/S}$. 
\end{notation}

\begin{definition}
If $G$ is a finite group acting on a gerbe $\cX$ or $\cY$ in the situation of Theorem~\ref{theorem-Aut-mun-gerbe}, we say that $G$ \emph{acts by gerbe automorphisms} if for every $g \in G$, the corresponding automorphism acts as the identity on stabilizer groups. 
\end{definition}

\begin{remark}
\label{remark-group-actions}
    The notion of an action of a group $G$ on an algebraic stack $\cX$ is somewhat subtle, because the category of algebraic stacks naturally forms a $2$-category. 
    Similarly, if $\cC$ is a linear category, such as a twisted derived category $\cC = \Dperf(X, \alpha)$, then group actions must be understood in the $\infty$-categorical sense. 
    In this paper, the term ``group action'' always refers to the natural higher-categorical notion, discussed in detail in \cite[\S3.1]{bayer-perry} (see also \cite{romagny} for the case of stacks). 
    Nonetheless, for emphasis we sometimes speak of a \emph{$2$-categorical action} of $G$ on a stack $\cX$, or an \emph{$\infty$-categorical action} of $G$ on a linear category $\cC$. 
    
    Any such action also gives rise to a \emph{1-categorical action} of $G$, which is the data of a homomorphism from $G$ to the relevant group of automorphisms (modulo equivalences of automorphisms). 
    For instance, if $\psi$ is an action of a finite group $G$ on a $\bmu_n$-gerbe $\cX$ by gerbe automorphisms as in Theorem~\ref{theorem-Aut-mun-gerbe}, then the associated $1$-categorical action is a homomorphism 
    \begin{equation*}
        \psi_1 \colon G \to \pi_0(\cAut_{\cX/S}(S)) ,
    \end{equation*}
    where $\pi_0(\cAut_{\cX/S}(S))$ denotes the group of isomorphism classes of objects in the groupoid $\cAut_{\cX/S}(S)$. 
    Similarly, if $\phi$ is an action of $G$ on $\Dperf(X, \alpha)$ as an $S$-linear category, then the associated $1$-categorical action is a homomorphism 
    \begin{equation*}
        \phi_1 \colon G \to \pi_0(\cAut_{\Dperf(X,\alpha)/S}(S)). 
    \end{equation*}
\end{remark}

\begin{remark}
\label{remark-variant-1-categorical-action}
    Let us describe a slightly different perspective on a $1$-categorical action. 
    We begin with a general fact. 
    Let $\cM \colon (\Sch/S)^{\op} \to \Grpd$ be a stack with respect to the \'{e}tale topology.  
    We denote by 
    \begin{equation*}
        \pi_0(\cM) \colon (\Sch/S)^{\op} \to \Set 
    \end{equation*}
    the presheaf $T \mapsto \pi_0(\cM(T))$, where $\pi_0(\cM(T))$ denotes the set of isomorphism classes of the groupoid $\cM(T)$, and by 
    \begin{equation*}
        \pi_0^{\mathrm{sh}}(\cM) \colon (\Sch/S)^{\op} \to \Set 
    \end{equation*} 
    the \'{e}tale sheafification of $\pi_0(\cM)$. 
    In general, if $\cM$ is algebraic the sheaf of sets $\pi_0^{\sh}(\cM)$ need not be algebraic, but if $\cM \to M$ is a gerbe over an algebraic space $M$ over $S$, then this is true and in fact $\pi_0^{\sh}(\cM) \cong M$ \cite[\href{https://stacks.math.columbia.edu/tag/06QD}{Tag 06QD}]{stacks-project}.

    Now suppose that  $\cX \to X$ is a $\bmu_n$-gerbe as in Theorem~\ref{theorem-Aut-mun-gerbe}. 
    Then $\cAut_{\cX/S} \to \Aut_{\cX/S}$ is a $\bmu_n$-gerbe, so that by the above 
    \begin{equation}
    \label{pi0sh-cAut}
        \pi_0^{\sh}(\cAut_{\cX/S}) \cong \Aut_{\cX/S} . 
    \end{equation} 
    In particular, there is a natural homomorphism 
    \begin{equation}
    \label{pi0-Aut}
        \pi_0(\cAut_{\cX/S}(S)) \to \Aut_{\cX/S}(S). 
    \end{equation}
    Given a $1$-categorical action 
    \begin{equation*}
        \psi_1 \colon G \to \pi_0(\cAut_{\cX/S}(S))
    \end{equation*}
    by composing we obtain a homomorphism 
    \begin{equation}
    \label{psi'_1}
        \psi'_1 \colon G \to \Aut_{\cX/S}(S) 
    \end{equation}
    Conversely, suppose that we are given a homomorphism 
    \begin{equation*}
        \psi'_1 \colon G \to \Aut_{\cX/S}(S). 
    \end{equation*}
    This does not automatically lift to a $1$-categorical action, since the homomorphism~\eqref{pi0-Aut} need not be an isomorphism. 
    However, given any point $s \in S$, it follows from~\eqref{pi0sh-cAut} that there exists an \'{e}tale neighorhood $U \to S$ of $s$ such that a lift exists after base change to $U$, i.e. there exists a $1$-categorical action 
    \begin{equation*}
        (\psi_U)_1 \colon G \to \pi_0(\cAut_{\cX_U/U}(U))
    \end{equation*}
    which induces the homomorphism 
    \begin{equation*}
     (\psi'_1)_U \colon G \xrightarrow{\, \psi'_1 \, } \Aut_{\cX/S}(S) \to \Aut_{\cX_U/U}(U). 
    \end{equation*} 
    In this sense, at least up to working \'{e}tale locally, there is no difference between a~$1$-categorical action and a homomorphism~\eqref{psi'_1}.  
    This argument also shows that if $S = \Spec(k)$ is the spectrum of an algebraically closed field, then there is no difference between the two notions, since~\eqref{pi0-Aut} is an isomorphism. 

    We note that the above remarks also apply to $1$-categorical actions on other objects, e.g. on the twisted derived category $\Dperf(X,\alpha)$. 
\end{remark}

\begin{warning}
    It is not true that a given $1$-categorical action must lift to a higher-categorical action. 
    The obstruction to the existence of such a lift is discussed in \cite[\S3]{bayer-perry} in general and in \cite[\S2.2]{oberdieck-equivariant} for linear categories. 
    This obstruction may be nontrivial even in very simple, natural examples. 
    For instance, \cite[\S3.6]{oberdieck-equivariant} shows that for an elliptic curve $E$, there exist $2$-torsion points $(a,b) \in E \times E^{\vee}$ such that if $t_{a}$ denotes translation by $a$ and $L_b$ the corresponding $2$-torsion line bundle, then the autoequivalence $\phi = t_a^*(-) \otimes L_b$ generates a $1$-categorical action of $\bZ/2$ on $\Dperf(E)$ that does not admit a $2$-categorical lift. 
\end{warning}

\begin{construction}[Action of $\cAut_{\cX/S}$ on $\Dperf(X,\alpha)$]
\label{remark-Aut-bmun-act-DXalpha}
    In the situation of Theorem~\ref{theorem-Aut-mun-gerbe}, let $\alpha \in \Br(X)$ be the Brauer class associated to the $\bmu_n$-gerbe $\cX \to X$ and assume that $n$ is invertible on $S$. 
    Recall (see e.g. \cite{bergh-BS}) that there is a completely orthogonal decomposition 
    \begin{equation*}
        \Dperf(\cX) = \langle \Dperf(X), \Dperf(X, \alpha), \Dperf(X, 2\alpha), \dots, \Dperf(X, (n-1)\alpha) \rangle. 
    \end{equation*}
    Since $\Dperf(X, i \alpha)$ is the subcategory of $\Dperf(\cX)$ where the inertia of $\cX$ acts with weight~$i$, we see that the action of any automorphism in $\cAut_{\cX/S}(S)$ preserves this semiorthogonal decomposition.

    Now suppose that $G$ is a finite group which acts on $\cX$ by gerbe automorphisms. 
    Then the corresponding action of $G$ on $\Dperf(\cX)$ preserves the above semiorthogonal decomposition. If $|G|$ is invertible on $S$, then taking invariants we find an $S$-linear completely orthogonal decomposition 
    \begin{equation}
    \label{sod-cX-mod-G}
        \Dperf(\cX/G) \simeq \Dperf(\cX)^G = \langle \Dperf(X)^G, \Dperf(X,\alpha)^G, \dots, \Dperf(X, (n-1)\alpha)^G \rangle.
    \end{equation}
    This shows that the invariant category $\Dperf(X, \alpha)^G$ is an $S$-linear semiorthogonal component of the derived category of the Deligne--Mumford quotient stack $\cX/G$. 
\end{construction}

\begin{remark}
    When $S = \Spec(k)$ is the spectrum of a field in which $n$ is invertible, 
    then it follows from the above construction and \cite{bergh-geometricity} that $\Dperf(X, \alpha)^G$ is of geometric origin. 
    In Lemma~\ref{lemma-DperfXalphaG-geometric} below, we will give a more explicit proof of this consequence, which depends on an additional assumption but has the benefit of generalizing to the relative situation. 
\end{remark}

\begin{construction}[Action of $\cAut_{\cY/S}$ on $\Dperf(X,\alpha)$]
    \label{remark-Aut-BGm-act-DXalpha}
    Similarly, in the situation of Theorem~\ref{theorem-Aut-mun-gerbe}, if $\alpha \in \Br(X)$ is the Brauer class associated to the $\bG_m$-gerbe $\cY \to X$, then any automorphism in $\cAut_{\cY/S}(S)$ acts naturally on $\Dperf(X, \alpha)$. 
    Moreover, for any action of a finite group $G$ on $\cY$ by gerbe automorphisms, 
    the invariant category $\Dperf(X, \alpha)^G$ is a semiorthogonal component of the quotient stack $\cY/G$. 
    We also note that the action of $\cAut_{\cX}$ on $\Dperf(X,\alpha)$ factors through the action of $\cAut_{\cY}$. 
\end{construction}

\begin{remark}
    In this section, we will largely focus on the story for $\bmu_n$-gerbes, instead of the simpler one for $\bG_m$-gerbes. 
    The reason is that, in good cases, for an action of a finite group $G$ on a $\bmu_n$-gerbe $\cX$ by gerbe automorphisms, we can prove that $\Dperf(\cX/G)$ --- and hence any of its semiorthogonal components, like $\Dperf(X, \alpha)^G$ --- is of geometric origin (see Lemmas~\ref{lemma-cX-mod-G-geometric} and~\ref{lemma-DperfXalphaG-geometric}). 
\end{remark}

\subsection{Lifting group actions on derived categories to gerbes} 
Now we explain how to geometrize certain actions on $\Dperf(X)$ in terms of the above constructions. 
First we focus on the absolute case, when the base is a field. 

\begin{lemma}
\label{lemma-geometrize-cY}
    Let $X$ be a smooth proper variety over a field $k$. 
    Let $G \subset \Aut_X(k) \ltimes \Pic_X(k)$ be a finite group which acts on $\Dperf(X)$ (through automorphisms and tensoring by line bundles). 
    Let $\cY \to X$ be the trivial $\bG_m$-gerbe. 
        Then there is an action of $G$ on $\cY$ by gerbe automorphisms which recovers via Construction~\ref{remark-Aut-BGm-act-DXalpha} the given $G$-action on $\Dperf(X)$.
\end{lemma}

\begin{proof}
    By \cite[Example 1.2]{olsson}, there is a natural splitting 
    \begin{equation*}
        \cAut_{\cY} \simeq \Aut_{X} \ltimes \cPic_X 
    \end{equation*}
    as group stacks. 
    The splitting respects the actions of each side on $\Dperf(X)$. 
    This implies the result. 
    To be more precise, recall that we regard all actions on $\Dperf(X)$ in the $\infty$-categorical sense. A priori the above argument produces a $G$-action on $\cY$ that induces a $G$-action on $\Dperf(X)$ which agrees with the given one at the $2$-categorical level, but it follows from \cite[Corollary 3.4]{bayer-perry} that this also holds at the $\infty$-categorical level. 
\end{proof}

Now we prove a similar result for lifting group actions to $\bmu_n$-gerbes. 
For applications, instead of taking a trivial $\bmu_n$-gerbe $\cX \to X$, it is useful to allow $\cX \to X$ to only be \emph{essentially trivial}, meaning that its pushout to a $\bG_m$-gerbe $\cY \to X$ is trivial. 

\begin{proposition}
\label{proposition-lifting-action-D-to-cX}
    Let $X$ be a smooth proper variety over an algebraically closed field $k$ with $\omega_X \cong \cO_X$. 
    Let $G \subset \Aut^0_X(k) \times \Pic^0_X(k)$ be a finite group which acts on $\Dperf(X)$ (through automorphisms and tensoring by line bundles). 
    Let us denote by $\phi$ the given $G$-action on $\Dperf(X)$ and by 
    \begin{equation*}
        \phi_1 \colon G \to \Aut_{\Dperf(X)}(k) 
    \end{equation*}
    the associated $1$-categorical action.  
    Let $n$ be a multiple of $|G|$ which is invertible in~$k$. 
    Let $\cX \to X$ be a $\bmu_n$-gerbe which is essentially trivial. 
    \begin{enumerate}
        \item \label{barpsi1-exists} 
        Up to possibly replacing $n$ by a multiple $m$ and $\cX$ with its pushout along $\bmu_n \to \bmu_{m}$,  
        there exists a homomorphism $\bar{\psi}_1 \colon G \to \Aut_{\cX}(k)$ which fits into a commutative diagram 
        \begin{equation}
        \label{barpsi1}
        \begin{tikzcd}
        & \Aut_{\cX}(k) \ar[d]  \\ 
        G \ar{r}[swap]{\phi_1} \ar{ur}{\bar{\psi}_1}  & \Aut_{\Dperf(X)}(k) ,
        \end{tikzcd}
    \end{equation}
    where the vertical arrow is given by Construction~\ref{remark-Aut-bmun-act-DXalpha}. 
        \item \label{G-through-cX}
        Given a homomorphism $\bar{\psi}_1$ as in~\eqref{barpsi1-exists}, 
        there exists a finite group $G'$, a surjection $\pi \colon G' \twoheadrightarrow G$, and an action $\psi$ of $G'$ on $\cX$ by gerbe automorphisms 
        such that the induced $1$-categorical action is given by the composition 
    \begin{equation}
    \label{psi1}
        \psi_1 \colon G' \xrightarrow{\, \pi \, } G \xrightarrow{ \, \bar{\psi}_1 \,} \Aut_{\cX}(k). 
    \end{equation}
        \end{enumerate}
        Assuming further that $|G'|$ is invertible in $k$, then:   
        \begin{enumerate}[resume]
        \item \label{G'-action-lifting-G-action-DX}
        Up to possibly replacing $n$ by a  multiple $m$ which is divisible by $|G'|$ and $\cX$ with its pushout along $\bmu_n \to \bmu_{m}$, we may find a $G'$-action $\psi$ on $\cX$ by gerbe automorphisms such that the induced $G'$-action on $\Dperf(X)$ via  Construction~\ref{remark-Aut-bmun-act-DXalpha} recovers the $G'$-action on $\Dperf(X)$ obtained from the given $G$-action $\phi$ on $\Dperf(X)$ by extension along $\pi \colon G' \to G$.
    \end{enumerate} 
\end{proposition}

\begin{proof}
\eqref{barpsi1-exists} Let $\cY \to X$ be the trivial $\bG_m$-gerbe given by the pushout of $\cX \to X$. 
By (the proof of) Lemma~\ref{lemma-geometrize-cY}, 
the inclusion $G \subset \Aut_{X}^0(k) \times \Pic^0_X(k)$ lifts to an injective homomorphism $G \hookrightarrow \Aut_{\cY}^0(k)$. 

Let $g \in G$ and regard it as an element of $\Aut^0_{\cY}(k)$ via this injection, and let $g_X \in \Aut_X(k)$ be its image under $\cAut_{\cY}(k) \to \Aut_X(k)$. 
By Theorem~\ref{theorem-Aut-mun-gerbe}\eqref{Aut-0-surjective}, we may find $\wtilde{g} \in \Aut^0_{\cX}(k)$ whose image $\wtilde{g}_X$ in $\Aut_{X}(k)$ is equal to $g_X$. 
Note that since $g_X$ has order dividing $|G|$, it follows from the top row of~\eqref{Pic-Aut-Aut} that $\wtilde{g}_X$ has order dividing $n + |G|$. 
If $\wtilde{g}_Y$ denotes the image of $\wtilde{g}$ under $\Aut_{\cX}^0(k) \to \Aut_{\cY}^0(k)$, then by construction $\wtilde{g}_Y g^{-1}$ is in the kernel of the map $\Aut_{\cY}(k) \to \Aut_{X}(k)$. 
Moreover, $\Aut^0_{\cY}(k)$ is an abelian variety by Theorem~\ref{theorem-Aut-mun-gerbe}\eqref{identity-components}, so the product $\wtilde{g}_Y g^{-1}$ must have order divisible by $n + |G|$, since its factors do. 
Thus, from the second row of~\eqref{Pic-Aut-Aut}, we find that $\wtilde{g}_Yg^{-1}$ is in the image of the injection $\Pic_{X}(k)[n + |G|](k) \to \Aut_{\cY}(k)$. 
Now choose $m = dn$ to be a multiple of $n$ where $d>1$ is invertible in $k$ (so that $m$ is invertible in $k$). 
Let $\cX'$ be the pushout of $\cX$ along $\bmu_n \to \bmu_{m}$, and let $\wtilde{g}_{\cX'}$ denote the image of $\wtilde{g}$ under $\Aut^0_{\cX}(k) \to \Aut^0_{\cX'}(k)$. 
Then the image of $\wtilde{g}_{\cX'}$ in $\Aut^0_{\cY}(k)$ is still $\wtilde{g}_Y$, so it differs from $g$ by an element $L \in \Pic_{X}(k)[n + |G|](k)$. 
Since $n + |G| \leq dn$ by our assumption that $n$ is a multiple of $|G|$, the element $L$ is contained in $\Aut_{\cX'}(k)$ by the top row of~\eqref{Pic-Aut-Aut} for $\cX'$. 
Thus, we find that $\wtilde{g}_{\cX'} L^{-1}$ maps to $g$ under $\Aut_{\cX}(k) \to \Aut_{\cY}(k)$.

From now on, we replace $\cX$ with $\cX'$. 
Let $\wtilde{G} \subset \Aut_{\cX}(k)$ be the preimage of $G$ along the homomorphism $\Aut_{\cX}(k) \to \Aut_{\cY}(k)$.
This map is surjective by the previous paragraph, and injective by~\eqref{Pic-Aut-Aut}, so it is an isomorphism. 
Thus we have produced a homomorphism $\bar{\psi}_1 \colon G \to \Aut_{\cX}(k)$ fitting into the commutative diagram~\eqref{barpsi1}, as required in~\eqref{barpsi1-exists}. 
    
\eqref{G-through-cX}    Since $\cAut_{\cX} \to \Aut_{\cX}$ is a $\bmu_n$-gerbe, the obstruction to lifting the $1$-categorical action $\bar{\psi}_1$ of $G$ on $\cX$ to a genuine $2$-categorical action (i.e. an action on the stack $\cX$) is an element $\ob(\bar{\psi}_1) \in \rH^3(G, \bmu_n)$, and when this obstruction vanishes the set of lifts is a $\rH^2(G, \bmu_n)$-torsor (see \cite[Lemma~3.2]{bayer-perry}). 
    By (the proof of) \cite[Theorem 2.1]{oberdieck-equivariant}, we may find a finite group $G'$ and a surjection $\pi \colon G' \twoheadrightarrow G$ such that the map $\pi^* \colon \rH^3(G, \bmu_n) \to \rH^3(G', \bmu_n)$ kills $\ob(\bar{\psi}_1)$. 
    But $\pi^*(\ob(\bar{\psi}_1))$ is precisely the obstruction to lifting the homomorphism 
    \begin{equation*}
        \psi_1 \colon G' \xrightarrow{\pi} G \to \Aut_{\cX}(k)
    \end{equation*} 
    to a $2$-categorical action of $G'$ on $\cX$, so such a $G'$-action on $\cX$ exists and the collection of such forms a $\rH^2(G',\bmu_n)$-torsor. 

\eqref{G'-action-lifting-G-action-DX} 
Note that if we replace $\cX$ with its pushout along $\bmu_n \to \bmu_m$ where $m$ is divisible by $n$, then the existence statements in~\eqref{barpsi1-exists} and~\eqref{G-through-cX} is preserved; up to such a replacement, from now on we may thus assume that $n$ is a multiple of $|G'|$. 

Let $\pi^*(\phi)$ denote the $G'$-action on $\Dperf(X)$ given by the extension of $\phi$ along $\pi \colon G' \to G$. 
Its $1$-categorical action is given by 
\begin{equation*}
    \pi^*(\phi_1) = \phi_1 \circ \pi \colon G' \to \Aut_{\Dperf(X)}(k). 
\end{equation*}
The collection of $G'$-actions on $\Dperf(X)$ which lift the $1$-categorical action $\pi^*(\phi_1)$ forms a nonempty $\rH^2(G', k^{\times})$-torsor (see \cite[Corollary~3.4]{bayer-perry}). 
Similarly, consider the collection of $G'$-actions $\psi$ on $\cX$ by gerbe automorphisms which lift the $1$-categorical action $\psi_1$ from~\eqref{psi1}. 
By \eqref{G-through-cX}, they form a nonempty $\rH^2(G', \bmu_n)$-torsor, and given such a $\psi$, the induced action $\psi_{\Dperf(X)}$ on $\Dperf(X)$ via Construction~\ref{remark-Aut-bmun-act-DXalpha} is a lift of the $1$-categorical action $\pi^*(\phi_1)$. 
The induced map of torsors is compatible with the natural map 
\begin{equation*}
    \rH^2(G', \bmu_n) \to \rH^2(G', k^{\times}). 
\end{equation*}
Therefore, in order for a $G'$-action $\psi$ on $\cX$ as in~\eqref{G'-action-lifting-G-action-DX} to exist, 
it suffices for the above map to be surjective. 
This follows from the long exact sequence on cohomology associated to the sequence
    \begin{equation*}
        1 \to \bmu_n \to k^{\times} \xrightarrow{ (-)^n} k^{\times} \to 1, 
    \end{equation*}
    since $n$ annihilates $\rH^{>0}(G', k^{\times})$ by our assumption that it is a multiple of $|G'|$.
\end{proof}

\subsection{Deforming group actions} 

The above results allow us to lift group actions on the derived category to actions on a gerbe. 
Now we want to explain how to make such actions in a family of categories, by deforming such a lift on the special fiber. 
To do so, we will use some results from \cite[\S3.2.3]{bayer-perry}, which show how to deform actions on the derived category up to passing to an \'{e}tale neighborhood. 
We will also need similar results for deforming actions on a $\bmu_n$-gerbe. 

A ring $A$ is called a \emph{Grothendieck ring} if it is noetherian and for every $\fp \in \Spec A$, the completion $A_{\fp} \to \widehat{A_{\fp}}$ of the local ring at $\fp$ is a regular map of rings \cite[\href{https://stacks.math.columbia.edu/tag/07GG}{Tag 07GG}]{stacks-project}. 
A scheme $S$ is called a \emph{Grothendieck scheme} if for every open affine $U \subset S$, the ring $\cO_S(U)$ is Grothendieck; this is a mild condition, which includes all excellent schemes. 

\begin{proposition}
\label{proposition-deform-group-actions}
    Let $S$ be a Grothendieck scheme. 
    Let $X \to S$ be a smooth proper morphism of schemes with geometrically integral fibers. 
    Let $\cX \to X$ be a $\bmu_n$-gerbe with associated Brauer class $\alpha \in \Br(X)$. 
    Fix a point $s \in S$. 
    \begin{enumerate}
        \item \label{deform-G-DX} Let $G$ be a finite group whose order is invertible in $\kappa(s)$, and let 
        \begin{equation*}
         \phi_1 \colon G \to \pi_0(\cAut_{\Dperf(X, \alpha)/S}(S))    
        \end{equation*}
        be a $1$-categorical action of $G$ on $\Dperf(X,\alpha)$. 
        Assume that the obstruction class $\ob((\phi_1)_s) \in \rH^3(G, \kappa(s)^{\times})$ to the existence of an $\infty$-categorical lift of the $1$-categorical action $(\phi_1)_s$ on $\Dperf(X_s, \alpha_s)$ vanishes. 
        Then there exists an \'{e}tale neighborhood $U \to S$ of $s$ such that the obstruction $\ob((\phi_1)_U) \in \rH^3(G, \bG_m(U))$ vanishes, and thus the set of $\infty$-categorical lifts of $(\phi_1)_U$ is a nonempty $\rH^2(G, \bG_m(U))$-torsor. 
        
        \item \label{deform-G-cX} Assume that $n$ is invertible on $S$, let $G'$ be a finite group, and let 
        \begin{equation*}
        \psi_1 \colon G' \to \pi_0(\cAut_{\cX/S}(S))
        \end{equation*}
        be a $1$-categorical action of $G'$ on $\cX$ by gerbe automorphisms. 
        Assume that the obstruction class $\ob((\psi_1)_s) \in \rH^3(G', \bmu_n(\kappa(s)))$ to the existence of a $2$-categorical lift of the $1$-categorical action $(\psi_1)_s$ on the fiber $\cX_s$ vanishes. 
        Then there exists an \'{e}tale neighborhood $U \to S$ of $s$ such that the obstruction $\ob((\psi_1)_{U}) \in \rH^3(G', \bmu_n(U))$ vanishes, and thus the set of $2$-categorial lifts of $(\psi_1)_U$ is a nonempty $\rH^2(G', \bmu_n(U))$-torsor. 
        
        \item \label{deform-G-cX-DX} 
        Suppose the assumptions in parts~\eqref{deform-G-DX} and~\eqref{deform-G-cX} above both hold and  that the order of $G'$ is invertible in $\kappa(s)$. 
        Moreover, assume that $\phi_s$ is a given $\infty$-categorical $G$-action on $\Dperf(X_s, \alpha_s)$ which lifts the $1$-categorical action 
        \begin{equation*}
        (\phi_1)_s \colon G \xrightarrow{\, \phi_1 \,} \pi_0(\cAut_{\Dperf(X, \alpha)/S}(S)) \to \pi_0(\cAut_{\Dperf(X_s, \alpha_s)/\kappa(s)}(\kappa(s))),
        \end{equation*} 
        that $\psi_s$ is a $2$-categorical $G'$-action on $\cX_s$ by gerbe automorphisms which lifts 
        \begin{equation*}
        (\psi_1)_s \colon G' \xrightarrow{\, \psi_1 \,} \pi_0(\cAut_{\cX/S}(S)) \to \pi_0(\cAut_{\cX_s/\kappa(s)}(\kappa(s))) , 
        \end{equation*} 
        and that $\pi \colon G' \to G$ is a surjection such that the $G'$-action on $\Dperf(X_s, \alpha_s)$ induced by $\psi_s$ via Construction~\ref{remark-Aut-bmun-act-DXalpha} recovers the extension of $\phi_s$ along $\pi \colon G' \to G$. 
	Finally, assume that $\phi_1$ and $\psi_1$ are compatible in the sense that the diagram 
        \begin{equation}
        \label{G'-actions-compatible}
        \begin{tikzcd}
        & & \pi_0(\cAut_{\cX/S}(S)) \ar[d]  \\ 
        G' \ar{r}[swap]{\pi} \ar{urr}{{\psi}_1} & G \ar{r}[swap]{\phi_1}   & \pi_0(\cAut_{\Dperf(X,\alpha)/S}(S)) ,
        \end{tikzcd}
    \end{equation}
    commutes, where the vertical arrow is given by Construction~\ref{remark-Aut-bmun-act-DXalpha}. 
        Then we may find an \'{e}tale neighborhood $(U, u) \to (S,s)$ of $s$, an $\infty$-categorical action $\phi_U$ of $G$ on $\Dperf(X_U, \alpha_U)$, and a $2$-categorical action $\psi_U$ of $G'$ on $\cX_U$ by gerbe automorphisms, such that: 
        \begin{enumerate}
            \item \label{phis-deform}
            The restriction $\phi_u$ of $\phi_U$ to the fiber $\Dperf(X_u, \alpha_u)$ coincides with (the base change to $u$ of)~$\phi_s$. 
            \item \label{psis-deform}
            The restriction $\psi_u$ of $\psi_U$ to the fiber $\cX_u$ coincides with (the base change to $u$ of)~$\psi_s$. 
            \item \label{psiU-phiU-compatible} The $G'$-action on $\Dperf(X_U, \alpha_U)$ induced by $\psi_U$ via Construction~\ref{remark-Aut-bmun-act-DXalpha} recovers the extension of $\phi_U$ along $\pi \colon G' \to G$. 
        \end{enumerate}
    \end{enumerate}
\end{proposition}

\begin{proof}
    The claim~\eqref{deform-G-DX} is precisely \cite[Proposition 3.9]{bayer-perry}. 
    
    For~\eqref{deform-G-cX}, we can argue similarly. 
    Namely, the statement reduces to proving the following $\bmu_n$-analog of \cite[Lemma 3.10]{bayer-perry}: Suppose that $A$ is a Grothendieck ring on which $n$ is invertible, $\alpha \in \rH^d(G', \bmu_n(A))$ is a degree $d \geq 1$ cohomology class (where $G'$ acts trivially on $\bmu_n(A)$), and $\fp \in \Spec A$ is a point such that the map $\rH^d(G', \bmu_n(A)) \to \rH^d(G', \bmu_n(\kappa(\fp)))$ kills $\alpha$. 
    Then we claim there exists an affine \'{e}tale neighborhood $\Spec B \to \Spec A$ of $\fp$ such that the map $\rH^d(G', \bmu_n(A)) \to \rH^d(G', \bmu_n(B))$ kills $\alpha$. 
    As in the proof of \cite[Lemma~3.10]{bayer-perry}, this in turn reduces to proving 
    if $A$ is a complete local ring with residue field $\kappa$, then for $d \geq 1$ the map $\rH^d(G', \bmu_n(A)) \to \rH^d(G', \bmu_n(\kappa))$ is an isomorphism. 
    The same proof as in \cite[Lemma~3.11]{bayer-perry} works, using that for a surjection $A \to B$ of artinian local rings with the same residue field, the map $\bmu_n(A) \to \bmu_n(B)$ is an isomorphism because $n$ is invertible.

    Finally, we prove~\eqref{deform-G-cX-DX}. We may assume $S = \Spec A$ is affine and $s = \fp \in \Spec A$. 
    Because $A$ is a Grothendieck ring, the composition $A \to A_{\fp} \to \widehat{A_{\fp}}$ is a regular map of noetherian rings, so by Popescu's theorem \cite[\href{https://stacks.math.columbia.edu/tag/07GC}{Tag 07GC}]{stacks-project} we may write $\widehat{A_{\fp}} = \colim A_i$ as a filtered colimit of smooth ring maps $A \to A_i$. 
    Let $d \geq 1$ be an integer and consider the commutative diagram 
        \begin{equation*}
        \begin{tikzcd}
        \rH^d(G', \bmu_n(A_i)) \ar{d} \ar{r} & \rH^d(G', \bmu_n(\kappa(\fp))) \ar{d} \\ 
        \rH^d(G', \bG_m(A_i)) \ar{r} & \rH^d(G', \kappa(\fp)^{\times}) \\ 
        \rH^d(G, \bG_m(A_i)) \ar{r} \ar{u} & \rH^d(G, \kappa(\fp)^{\times}) \ar{u}. 
        \end{tikzcd}
        \end{equation*}
The colimit of the first column over $i$ is isomorphic to the second column. 
Indeed, note that $\colim \rH^d(G', \bmu_n(A_i)) = \rH^d(G', \bmu_n(\widehat{A_{\fp}}))$ as $\colim \bmu_n(A_i) = \bmu_n(\widehat{A_{\fp}})$. On the other hand, $\rH^d(G', \bmu_n(\widehat{A_{\fp}})) \cong \rH^d(G', \bmu_n(\kappa(\fp))$ by the previous paragraph. 
This proves our claim about the colimit for the first row. 
For the second and third rows, the claim follows from the same argument for $\bG_m$ in place of $\bmu_n$, by invoking~\cite[Lemma~3.11]{bayer-perry}. 

For $d = 3$, the top horizontal arrow sends the obstruction $\ob((\psi_1)_{A_i})$ to the existence of a $2$-categorical lift of $(\psi_1)_{A_i}$ to the obstruction $\ob((\psi_1)_{\fp})$, which vanishes by assumption. 
Therefore, we may choose $i$ so that $\ob((\psi_1)_{A_i}) = 0$, and thus there exists a $2$-categorical lift $\psi_{A_i}$ of $(\psi_1)_{A_i}$. 
Similarly, we may also assume that there exists an $\infty$-categorical lift $\phi_{A_i}$ of $(\phi_1)_{A_i}$. 
Further, we claim that we may arrange that the fiber $(\psi_{A_i})_{\fp}$ over $\fp$ is the given lift $\psi_{\fp}$. 
Indeed, consider the difference $(\psi_{A_i})_{\fp} - \psi_{\fp} \in \rH^2(G', \bmu_n(\kappa(\fp)))$. 
By the previous paragraph, up to enlarging $i$, this class is in the image of the top horizontal arrow for $d = 2$; hence we can modify $\psi_{A_i}$ by a lift of this difference along the arrow to ensure that $(\psi_{A_i})_{\fp} = \psi_{\fp}$. 
Similarly, we may also assume that the fiber $(\phi_{A_i})_{\fp}$ over $\fp$ is the given lift $\phi_{\fp}$ of $(\phi_1)_{\fp}$. 
Now let $(\psi_{A_i})_{\Dperf(X_{A_i}, \alpha_{A_i})}$ denote the induced $G'$-action on $\Dperf(X_{A_i}, \alpha_{A_i})$ via Construction~\ref{remark-Aut-bmun-act-DXalpha}, and let $\pi^*(\phi_{A_i})$ be the extension of $\phi_{A_i}$ along $\pi \colon G' \to G$. 
Note that both of these $G'$-actions have the same underlying $1$-categorical action, by the commutativity of the diagram~\eqref{G'-actions-compatible}.
Hence we may consider the difference 
\begin{equation*}
    (\psi_{A_i})_{\Dperf(X_{A_i}, \alpha_{A_i})} - \pi^*(\phi_{A_i}) \in \rH^2(G', \bG_m(A_i)),  
\end{equation*}
which maps to zero under the middle horizontal arrow in the above diagram for $d = 2$. 
Hence, up to enlarging $i$, we may assume that the difference itself is zero, i.e. $(\psi_{A_i})_{\Dperf(X_{A_i}, \alpha_{A_i})}$ coincides with $\pi^*(\phi_{A_i})$. 

The map $\Spec A_i \to \Spec A$ is smooth and its image contains $\fp$.
Hence we can find an \'{e}tale neighborhood $\Spec(B) \to \Spec(A)$ of $\fp$ over which $\Spec(A_i) \to \Spec(A)$ has a section, i.e. such that the map $A \to B$ factors through $A \to A_i$. 
Base changing the data constructed in the previous paragraph to $U = \Spec(B)$ gives $\psi_U$ and $\phi_U$ which satisfy~\eqref{phis-deform},~\eqref{psis-deform}, and~\eqref{psiU-phiU-compatible}. 
\end{proof}

The above essentially reduces the problem of constructing compatible actions on derived categories and gerbes in families to the $1$-categorical case, plus an enhanced construction on the special fiber. 
These problems can both be solved under a Calabi--Yau hypothesis. 

\begin{corollary}
\label{corollary-lifting-actions-in-CY-family}
Let $S$ is a variety over an algebraically closed field $k$ of characteristic $0$. 
Let $X \to S$ be a smooth proper morphism with geometrically integral fibers such  that $\omega_{X/S}$ is the pullback of a line bundle on $S$. 
Let $0 \in S(k)$ be a point. 
Let $G \subset \Aut^0_{X_0}(k) \times \Pic^0_{X_0}(k)$ be a finite group which acts on $\Dperf(X_0)$ (through automorphisms and tensoring by line bundles). 
Let $\cX \to X$ be a $\bmu_n$-gerbe whose base change $\cX_0 \to X_0$ over $0 \in S(k)$ is an essentially trivial $\bmu_n$-gerbe. 
Then up to replacing $\cX$ with its pushout along $\bmu_n \to \bmu_m$ for a suitable integer $m$ divisible by $n$ and up to base change to an \'{e}tale neighborhood of $0 \in S(k)$, there exists: 
\begin{enumerate}
    \item \label{G-action-deform-DX0}
    A $G$-action on the $S$-linear category $\Dperf(X,\alpha)$ whose base change to $0 \in S(k)$ recovers the given $G$-action on $\Dperf(X_0, \alpha_0) = \Dperf(X_0)$. 
    \item A finite group $G'$, a surjection $\pi \colon G' \twoheadrightarrow G$, and a $G'$-action on $\cX$ by gerbe automorphisms such that: 
    \begin{enumerate}
        \item the induced $G'$-action on $\Dperf(X,\alpha)$ via Construction~\ref{remark-Aut-bmun-act-DXalpha} recovers the $G'$-action on $\Dperf(X, \alpha)$ obtained from the $G$-action in~\eqref{G-action-deform-DX0} by extension along the homomorphism $\pi \colon G' \to G$; and 
    \item the induced homomorphism $G' \to \Aut_{X/S}(S)$ has image contained in the subgroup $\Aut_{X/S}^0[n](S)$. 
    \end{enumerate}
\end{enumerate}
\end{corollary}

\begin{proof}
    Let $\phi_0$ denote the given $G$-action on
    $\Dperf(X_0)$.
    By Proposition~\ref{proposition-lifting-action-D-to-cX}, up to possibly replacing $n$ with a multiple (and $\cX$ with its pushout), we may find a finite group $G'$, a surjection $\pi \colon G' \twoheadrightarrow G$, and a $G'$-action $\psi_0$ on $\cX_0$ by gerbe automorphisms such that the induced $G'$-action on $\Dperf(X_0)$ via Construction~\ref{remark-Aut-bmun-act-DXalpha} recovers the $G'$-action on $\Dperf(X_0)$ obtained from the $G$-action $\phi_0$ by extension along $\pi \colon G' \to G$. 

    Note that by Theorem~\ref{theorem-Aut-mun-gerbe}\eqref{identity-components}, the identity component $\Aut_{X/S}^0$ exists and is a smooth proper abelian scheme over $S$. 
    Its $n$-torsion $\Aut^0_{X/S}[n]$ is thus a finite \'{e}tale group scheme over $S$. 
    Consider the preimage $\Gamma \subset \Aut_{\cX/S}$ of $\Aut^0_{X/S}[n]$ under the morphism $\Aut_{\cX/S} \to \Aut_{X/S}$. 
     By parts~\eqref{Aut-ses} and \eqref{Aut-0-surjective} of Theorem~\ref{theorem-Aut-mun-gerbe}, $\Gamma$ fits into an exact sequence 
     \begin{equation*}
         1 \to \Pic_{X/S}[n] \to \Gamma \to \Aut^0_{X/S}[n] \to 1. 
     \end{equation*}
    Note that $\Pic_{X/S}[n]$ is also a finite \'{e}tale group scheme, because it can be written as an extension of such, 
    \begin{equation*}
        0 \to \Pic^0_{X/S}[n] \to \Pic_{X/S}[n] \to \NS_{X/S}[n] \to 0, 
    \end{equation*}
    where $\NS_{X/S}$ is the relative N\'{e}ron--Severi group. 
     Altogether, we find that $\Gamma$ is a finite \'{e}tale group scheme over $S$. 
    
    On the other hand, consider the group algebraic space $\Aut_{\Dperf(X,\alpha)/S}$ of autoequivalences of $\Dperf(X,\alpha)$ over $S$. 
    By \cite[Lemma~8.14]{PI-abelian3folds}, its identity component $\Aut_{\Dperf(X,\alpha)/S}^0$ exists, and is a smooth proper abelian scheme over $S$. Hence its $n$-torsion points $\Aut_{\Dperf(X,\alpha)/S}^0[n]$ is also a finite \'{e}tale group scheme over $S$. 

    Let $(\phi_0)_1$ and $(\psi_0)_1$ be the $1$-categorical actions underlying the actions $\phi_0$ and $\psi_0$. 
    By construction, they fit into a commutative diagram \begin{equation*}
        \begin{tikzcd}
        & & \Aut_{\cX_0}(k) \ar[d]  \\ 
        G' \ar{r}[swap]{\pi} \ar{urr}{({\psi_0})_1} & G \ar{r}[swap]{(\phi_0)_1}   & \Aut_{\Dperf(X_0,\alpha_0)}(k) . 
        \end{tikzcd}
    \end{equation*}
    Note that $\Aut_{\cX_0}(k) = \Aut_{\cX/S}(0)$ and $\Aut_{\Dperf(X_0,\alpha_0)}(k) = \Aut_{\Dperf(X, \alpha)/S}(0)$. Moreover, by construction $(\psi_0)_1$ factors through $\Gamma(0) \subset \Aut_{\cX/S}(0)$. 
    Therefore, since as explained above $\Gamma$ and $\Aut_{\Dperf(X,\alpha)/S}^0[n]$ are finite \'{e}tale group schemes, up to passing to an \'{e}tale neighborhood of $0 \in S(k)$, we may assume that we have a commutative diagram 
        \begin{equation*}
        \begin{tikzcd}
        & & \Aut_{\cX/S}(S) \ar[d]  \\ 
        G' \ar{r}[swap]{\pi} \ar{urr}{{\psi}_1} & G \ar{r}[swap]{\phi_1}   & \Aut_{\Dperf(X,\alpha)/S}(S) ,
        \end{tikzcd}
    \end{equation*}
    which recovers the above diagram over $0 \in S(k)$ and where $\psi_1$ factors through $\Gamma(S)$. 
    By Remark~\ref{remark-variant-1-categorical-action}, up to passing to a further \'{e}tale neighborhood of $0 \in S(k)$, we may assume that we in fact have a commutative diagram as in~\eqref{G'-actions-compatible}. Then we may apply Proposition~\ref{proposition-deform-group-actions}\eqref{deform-G-cX-DX} to conclude the proof. 
\end{proof}

\subsection{Geometricity of invariant categories}

To understand invariant categories of twisted derived categories, the results of the previous sections reduce us in good situations to understanding derived categories of quotients of $\bmu_n$-gerbes. Under appropriate hypotheses, the latter is in fact equivalent to a twisted derived category. 

\begin{lemma}
\label{lemma-cX-mod-G-geometric}
    Let $X \to S$ be a smooth proper morphism of schemes with geometrically integral fibers. 
    Let $\cX \to X$ be a $\bmu_n$-gerbe. 
    Let $G$ be a finite group of order invertible on $S$ which acts on $\cX$ by gerbe automorphisms. 
    Write $G$ as an extension 
    \begin{equation}
     \label{KGH}
        1 \to K \to G \to H \to 1 
    \end{equation}
    where $H$ is the image of $G$ in $\Aut_{X/S}(S)$. 
    Assume that the induced $H$-action on $X$ is free. 
    Then there exists a smooth proper morphism $M \to S$ with a surjective finite \'{e}tale morphism $M \to X/H$ over $S$ and a Brauer class $\beta \in \Br(M)$ such that 
    there is an $X/H$-linear equivalence 
    \begin{equation*}
    \Dperf(\cX/G) \simeq \Dperf(M, \beta).
    \end{equation*}
\end{lemma}

\begin{proof}
    The morphism $\cX \to X$ is equivariant with respect to the homomorphism $G \to H$, so we obtain an induced morphism $\cX/G \to X/H$. 
    The quotient $X/H$ is a smooth proper scheme over $S$ because of our assumption that $H$ acts freely on $X$. 
    Moreover, it follows from the fact that $G$ acts by gerbe automorphisms that $\cX/G \to X/H$ is a $\bmu_n \times K$-gerbe. 
    Then it follows from \cite[Proposition~2.12]{hotchkiss-pi} that there exists a finite \'{e}tale cover $M \to X/H$ and a Brauer class $\beta \in \Br(Y)$ with the desired property. 
\end{proof}

Now we track what happens to the $G$-invariants of the twisted derived category $\Dperf(X, \alpha)$ under the above equivalence. 
More generally, we explain what happens when the $G$-action on $\Dperf(X, \alpha)$ factors through a quotient $G \twoheadrightarrow \bar{G}$, since this situation arises naturally via the application of our earlier results, such as Corollary~\ref{corollary-lifting-actions-in-CY-family}.\footnote{Note that there is a mismatch in notation: $G$ in the present discussion plays the role of $G'$ in Corollary~\ref{corollary-lifting-actions-in-CY-family}, while $\bar{G}$ plays the role of $G$.}

\begin{proposition}
\label{lemma-DperfXalphaG-geometric}
    In the situation of Lemma~\ref{lemma-cX-mod-G-geometric}, assume that $X$ is a quasi-projective variety over an algebraically closed field $k$ and that $n$ is invertible in $k$. 
    Let $\alpha \in \Br(X)$ be the Brauer class associated to $\cX \to X$. 
    Assume that $\pi \colon G \twoheadrightarrow \bar{G}$ is a surjective homomorphism such that the induced $G$-action on $\Dperf(X, \alpha)$ via Construction~\ref{remark-Aut-bmun-act-DXalpha} factors through a $\bar{G}$-action on $\Dperf(X, \alpha)$. 
    \begin{enumerate}
        \item \label{DXalphaG-geometric-origin} The invariant categories $\Dperf(X, \alpha)^G$ and $\Dperf(X, \alpha)^{\bar{G}}$ are of geometric origin over~$S$. 
    \end{enumerate}
    Further assume that $S = \Spec(k)$ is a point. 
    Then there exists a decomposition 
    \begin{equation*}
    M = M_0 \sqcup M_1 \sqcup \cdots \sqcup M_{n-1}
    \end{equation*} 
    and a further decomposition 
    \begin{equation*}
        M_1 = N \sqcup N' 
    \end{equation*}
    with the following properties for all $0 \leq i \leq n-1$: 
    \begin{enumerate}[resume]
        \item 
        \label{M-decompose-Mi} 
        The morphisms $M_i \to X/H$, $N \to X/H$, and $N' \to X/H$ are finite \'{e}tale covers. 
        Moreover, if $k$ has characteristic $0$, $\alpha = 0$, and $\bar{G} \subset \Aut_{X}(k) \ltimes \Pic_X(k)$, then $N$ is connected. 
        \item If $\beta_i$ denotes the restriction of $\beta$ to $M_i$ and 
        $\gamma$ the restriction to $N$, then there are $X/H$-linear equivalences 
        \begin{align}
        \label{DXialpha-G'}
            \Dperf(X, i\alpha)^{G} & \simeq \Dperf(M_i, \beta_i) , \\
        \label{DXialpha-G} 
        \Dperf(X, \alpha)^{\bar{G}} & \simeq \Dperf(N, \gamma). 
        \end{align}
    \item \label{Gvee-act-cMi}
    Let $\cM_i \to M_i$ be the $\bG_m$-gerbe of class $\beta_i$. 
There is an action of $G^{\vee} \coloneqq \Hom(G, k^{\times})$ on $\cM_i$ by gerbe automorphisms such that the induced $G^{\vee}$-action on $\Dperf(M_i, \beta_i)$ via Construction~\ref{remark-Aut-BGm-act-DXalpha} corresponds to the residual $G^{\vee}$-action on $\Dperf(X, i\alpha)^{G}$ under the equivalence~\eqref{DXialpha-G'}. 

Let $\cN \to N$ be the $\bG_m$-gerbe of class $\gamma$ (i.e. the restriction of $\cM_{1}$ to $N$). 
Then the restriction of the above $G^{\vee}$-action on $\cM_1$ along $\pi^{\vee} \colon \bar{G}^{\vee} \to G^{\vee}$ preserves $\cN$, and hence induces a $\bar{G}^{\vee}$-action on $\cN$ by gerbe automorphisms. 
The induced $\bar{G}^{\vee}$-action on $\Dperf(N, \gamma)$ via Construction~\ref{remark-Aut-BGm-act-DXalpha} corresponds to the residual $\bar{G}^{\vee}$-action on $\Dperf(X, \alpha)^{\bar{G}}$ under the equivalence~\eqref{DXialpha-G}.

\item \label{Gvee-action-cMi-decomposition} Decomposing $G^{\vee}$ with respect to the exact sequence 
        \begin{equation*}
            1 \to \Pic_{M_i}(k) \to \Aut_{\cM_i}(k) \to \Aut_{M_i}(k) 
        \end{equation*}
        of Theorem~\ref{theorem-Aut-mun-gerbe}\eqref{Aut-ses}, i.e. projecting onto $\Aut_{M_i}(k)$, 
        yields an exact sequence 
        \begin{equation*}
            1 \to H^{\vee} \to G^{\vee} \to K^{\vee} 
        \end{equation*}
        which is dual to the sequence~\eqref{KGH}. 
        In these terms, $H^{\vee}$ acts by tensoring with line bundles pulled back from $X/H$, and $K^{\vee}$ acts by covering automorphisms of $M_i$ over $X/H$. 
    \end{enumerate}
\end{proposition}

\begin{proof}
    First we spell out some details about the equivalence $\Dperf(\cX/G) \simeq \Dperf(M, \beta)$. 
    By the construction from \cite[Proposition~2.12]{hotchkiss-pi} invoked in the proof of Lemma~\ref{lemma-cX-mod-G-geometric}, 
    $\cM \to X/H$ is the moduli stack of simple coherent sheaves on the gerbe $\cX/G \to X/H$. 
    This is naturally a $\bG_m$-gerbe $\cM \to M$ of class $\beta$, and the equivalence $\Dperf(\cX/G) \simeq \Dperf(M, \beta)$ is induced by the universal family on $\cX/G \times_{X/H} \cM$. 

    Note that we have a morphism $\cX/G \to BG$. Since the group $G^{\vee}$ corresponds to line bundles on $BG$, pulling back and tensoring we see that $G^{\vee}$ naturally acts on simple coherent sheaves on $\cX/G$, and hence on $\cM$. 
    By construction, it is straightforward to see that this $G^{\vee}$-action on $\cM$ is by gerbe automorphisms and satisfies the 
following properties: 
\begin{itemize}
    \item The induced $G^{\vee}$-action on $\Dperf(M, \beta)$ via Construction~\ref{remark-Aut-BGm-act-DXalpha} corresponds to the residual $G^{\vee}$-action on $\Dperf(\cX/G)$ under the equivalence of Lemma~\ref{lemma-cX-mod-G-geometric}. 
    \item Decomposing $G^{\vee}$ with respect to the exact sequence 
        \begin{equation*}
            1 \to \Pic_{M/S}(S) \to \Aut_{\cM/S}(S) \to \Aut_{M/S}(S) 
        \end{equation*}
        of Theorem~\ref{theorem-Aut-mun-gerbe}\eqref{Aut-ses} 
        yields an exact sequence 
        \begin{equation*}
            1 \to H^{\vee} \to G^{\vee} \to K^{\vee} 
        \end{equation*}
        which is dual to the sequence~\eqref{KGH}. 
        In these terms, $H^{\vee}$ acts by tensoring with line bundles pulled back from $X/H$, and $K^{\vee}$ acts by covering automorphisms of $M$ over $X/H$.
\end{itemize}

Furthermore, combining the semiorthogonal decomposition~\eqref{sod-cX-mod-G} with the equivalence of Lemma~\ref{lemma-cX-mod-G-geometric} gives an $X/H$-linear completely orthogonal decomposition 
\begin{equation}
\label{DMbeta-DXalphas}
    \Dperf(M, \beta) = \langle \Dperf(X)^{G}, \Dperf(X,\alpha)^{G}, \dots, \Dperf(X, (n-1)\alpha)^{G} \rangle. 
\end{equation}
Let $A = \ker(\pi)$, so that we have an exact sequence
 \begin{equation*}
     1 \to A \to G \to \bar{G} \to 1.
 \end{equation*}
Then by \cite[Lemma~3.6]{PS-crepant-res} we may compute $\Dperf(X,\alpha)^{G}$ by first taking $A$-invariants, and then $\bar{G}$-invariants. 
 By assumption, the $A$-action is trivial, so we obtain 
 \begin{equation}
 \label{DXalphaG}
     \Dperf(X, \alpha)^{G}  \simeq (\Dperf(X,  \alpha)^{A})^{\bar{G}} 
      \simeq (\Dperf(X, \alpha) \otimes \Dperf(BA))^{\bar G} . 
 \end{equation}
 Here, $\bar G$ acts on $\Dperf(X, \alpha)$ through the given $\bar{G}$-action, and on $\Dperf(BA)$ via covering automorphisms of the $\bar{G}$-cover $BA \to BG$. 
 Let $V_0, \dots, V_{m-1}$ be the irreducible finite-dimensional representations of $A$ over $k$, with $V_0 = k$ the trivial representation. 
 Then we have a completely orthogonal decomposition $\Dperf(BA) = \langle V_0, \dots, V_{m-1} \rangle$. 
 In terms of this decomposition, $\bar{G}$ acts trivially on $V_0$, but in general nontrivially on its complement ${^\perp}\langle V_0 \rangle = \langle V_1, \dots, V_{m-1} \rangle$. 
 Thus if we set $\cD = \Dperf(X,\alpha) \otimes {^\perp}\langle V_0 \rangle$, then we obtain a $\bar{G}$-equivariant completely orthogonal decomposition 
 \begin{equation*}
     \Dperf(X, \alpha) \otimes \Dperf(BA) = \langle \Dperf(X, \alpha) \otimes V_0, \cD \rangle. 
 \end{equation*}
 Taking $\bar{G}$-invariants and combining with~\eqref{DXalphaG}, we obtain a completely orthogonal decomposition 
 \begin{equation}
 \label{DXalphaG-barG}
     \Dperf(X,\alpha)^G \simeq \langle  \Dperf(X, \alpha)^{\bar G} \otimes V_0, \cD^{\bar G} \rangle. 
 \end{equation}

With the above preliminaries in place, we prove the statements in the proposition. 
Since $X$ is quasi-projective, so is the quotient $X/H$ and thus also the \'{e}tale cover $M \to X/H$. 
In particular, this implies that $\beta$ is represented by an Azumaya algebra, and hence $\Dperf(M, \beta)$ is of geometric origin over $S$. 
In view of~\eqref{DMbeta-DXalphas} and~\eqref{DXalphaG-barG}, the categories $\Dperf(X, \alpha)^G$ and $\Dperf(X, \alpha)^{\bar{G}}$ are both semiorthogonal components of $\Dperf(M,\beta)$, so they are also of geometric origin over $S$. 
This proves~\eqref{DXalphaG-geometric-origin}. 

Since the $X/H$-linear semiorthogonal decomposition~\eqref{DMbeta-DXalphas} is completely orthogonal, it must correspond to a decomposition of $M$ into 
$M_i$ which dominate $X/H$ (hence are also finite \'{e}tale over it) such that the equivalence~\eqref{DXialpha-G'} holds. 
Then the claims in~\eqref{Gvee-act-cMi} about $\cM_i$ and~\eqref{Gvee-action-cMi-decomposition} follow from the analogous properties for $\cM \to M$ stated above. 

Similarly, the completely orthogonal decomposition~\eqref{DXalphaG-barG} of $\Dperf(M_1, \beta_1) \simeq \Dperf(X,\alpha)^G$ corresponds to a decomposition $M_1 = N \sqcup N'$ where $N \to X/H$ and $N' \to X/H$ are finite \'{e}tale covers such that 
\begin{equation*}
    \Dperf(N, \gamma) \simeq \Dperf(X, \alpha)^{\bar{G}}   \quad \text{and} \quad 
    \Dperf(N', \gamma') \simeq \cD^{\bar{G}} , 
\end{equation*}
where $\gamma$ and $\gamma'$ denote the restriction of $\beta$ to $N$ and $N'$. 
The first of these equivalences gives~\eqref{DXialpha-G}, while the claims in~\eqref{Gvee-act-cMi} about $\cN \to N$ follow from the construction. 

All that remains to prove is the claim in~\eqref{M-decompose-Mi}  that if $k$ has characteristic $0$, $\alpha = 0$, and $\bar{G} \subset \Aut_{X}(k) \ltimes \Pic_X(k)$, then $N$ is connected. 
This is equivalent to the degree $0$ Hochschild cohomology of $N$ being a copy of the scalars, i.e. $\HH^0(\Dperf(N)) = k$. 
On the other hand, there are isomorphisms 
\begin{equation*}
    \HH^0(\Dperf(N)) \cong 
    \HH^0(\Dperf(N, \gamma)) \cong 
    \HH^0(\Dperf(X)^{\bar{G}})
\end{equation*}
where the first holds by \cite[Lemma~5.17]{PI-abelian3folds}\footnote{It is here that we use the assumption that $k$ has characteristic $0$, in order to ensure the hypothesis of \cite[Lemma~5.17]{PI-abelian3folds} that the period of $\gamma$ is invertible in $k$.} and the second by~\eqref{DXialpha-G} (and our assumption $\alpha = 0$). 
The final term can be computed using \cite[Theorem~4.4]{HH-group-actions}, which gives 
\begin{equation*}
    \textstyle \HH^0(\Dperf(X)^{\bar{G}}) \cong 
    \left( \bigoplus_{g \in \bar{G}} \HH^0(\Dperf(X), \phi_{g}) \right)^{\bar G}
\end{equation*}
where $\HH^0(\Dperf(X), \phi_{g})$ denotes Hochschild cohomology with coefficients in the autoequivalence $\phi_g$ corresponding to $g \in \bar{G}$. 
Write $g = (f, L) \in \Aut_{X}(k) \ltimes \Pic_X(k)$. 
If $f \neq \id_{X}$, then by our assumption that $H$ acts freely on $X$, it follows that the fixed locus of $f$ in $X$ is empty; then a computation shows that $\HH^0(\Dperf(X), \phi_g) = 0$, since it can be expressed in terms of coherent cohomology on this fixed locus (see for instance \cite[Proof of Theorem 3.2]{inner-product}). 
Similarly, a computation shows that if $g = (\id_{X}, L)$, then $\HH^0(\Dperf(X), \phi_g) \cong \rH^0(X, L)$; if $L$ is not the trivial line bundle, then since it is torsion this space must vanish. 
Altogether, this shows that the above formula reduces to 
\begin{equation*}
  \textstyle \HH^0(\Dperf(X)^{\bar{G}}) \cong 
    \HH^0(\Dperf(X))^{\bar G}. 
\end{equation*} 
Since $X$ is connected, $\HH^0(\Dperf(X)) = k$ and the $\bar{G}$-action on it is trivial. 
We conclude that 
$\HH^0(\Dperf(X)^{\bar{G}}) = k$, as required.
\end{proof}


\section{Semiregularity theorems} 
\label{section-semiregularity-theorems} 

In this section, we combine the above ingredients to prove our semiregularity theorems. 

\subsection{Noncommutative varieties}

\begin{theorem}
\label{theorem-semiregularity-categorical}
Let $\cC$ be a smooth proper $S$-linear category of geometric origin, where $S$ is a complex variety. 
Let $v$ be a section of $\Ktop[0](\cC/S)$.  
Let $0 \in S(\bC)$ be a point and let $E_0 \in \cC_0$ be a gluable semiregular object of class $v_0$. Then: 
\begin{enumerate}
    \item There exists an \'{e}tale morphism $U \to S$, a point $u \in U(\bC)$ mapping to $0 \in S(\bC)$, and an object $F \in \cC_{U}$ such that $F_{u} \simeq E_{0}$ (under the identification $\cC_{u} = \cC_0$). 
    \item For every point $s \in S(\bC)$, the fiber $v_s \in \Ktop[0](\cC_s)$ of $v$ is algebraic. 
\end{enumerate}
\end{theorem}

\begin{proof}
    In view of Theorem~\ref{theorem-semiregular-smooth}, the second statement is \cite[Proposition 8.1]{IHC-CY2}, while the first statement is shown in the course of the proof there. 
\end{proof}

\begin{remark}
\label{remark-semiregularity-categorical-rational}
    As in Remark~\ref{remark-semiregular-smooth-theorem}\eqref{semiregular-smooth-theorem-rational}, there is slight variant of Theorem~\ref{theorem-semiregularity-categorical} --- which holds by the same proof --- where we consider a section of $\Ktop[0](\cC/S)\otimes \bQ$ instead of $\Ktop[0](\cC/S)$. 
\end{remark}

We can bootstrap from this to an equivariant semiregularity theorem: 

\begin{theorem}
    \label{theorem-semiregularity-categorical-equivariant}
Let $\cC$ be a smooth proper $S$-linear category of geometric origin, where $S$ is a complex variety. 
Let $G$ be a finite group which acts on $\cC$ over $S$ in such a way that $\cC^G$ is smooth and proper of geometric origin over $S$. 
Let $v$ be a section of $\Ktop[0](\cC/S) \otimes \bQ$.  
Let $0 \in S(\bC)$ be a point and let $E_0 \in \cC_0$ be a gluable weakly $G$-semiregular object of class $v_0$. 
Let $\wtilde{E}_0 \in \cC^G_0$ be a semiregular object such that $\Forg(\wtilde{E}_0) \simeq E_0$, and 
assume that there exists a section $\wtilde{v}$ of $\Ktop[0](\cC^G/S) \otimes \bQ$ such that the class of $\wtilde{E}_0$ is $\wtilde{v}_0$. 
Then: 
\begin{enumerate}
    \item There exists an \'{e}tale morphism $U \to S$, a point $u \in U(\bC)$ mapping to $0 \in S(\bC)$, and an object $F \in \cC_{U}$ such that $F_{u} \simeq E_{0}$ (under the identification $\cC_{u} = \cC_0$). 
    \item For every point $s \in S(\bC)$, the fiber $v_s \in \Ktop[0](\cC/S) \otimes \bQ$ of $v$ is algebraic. 
\end{enumerate}
\end{theorem}

\begin{proof}
    Note that $\wtilde{E}_0$ is gluable, since $E_0$ is gluable and $\Ext^i(\wtilde{E}_0, \wtilde{E}_0) \cong \Ext^i(E_0, E_0)^G$ for any $i \in \bZ$. 
    Hence we may apply Theorem~\ref{theorem-semiregularity-categorical-equivariant} (in the form of Remark~\ref{remark-semiregularity-categorical-rational}) to conclude that there exists an \'{e}tale neighborhood $U \to S$ of $0 \in S(\bC)$ over which $\wtilde{E}_0$ extends to an object $\wtilde{F}$, and that the class $\wtilde{v}_s \in \Ktop[0](\cC^G_s)$ is algebraic for every $s \in S(\bC)$. 

    Then $F = \Forg(\wtilde{F})$ gives the required deformation of $E_0$, and the class of $F$ is the section $\Forg(\wtilde{v}) = v$, so that $v_s$ is algebraic for every $s \in S(\bC)$. 
\end{proof}

\begin{remark}
\label{remark-equivariant-semiregulary-conditions}
    Theorem~\ref{theorem-semiregularity-categorical-equivariant} has some limitations: 
    \begin{enumerate}
        \item \label{CG-smooth-proper} For a given $G$-action, it may be difficult to verify that $\cC^G$ is of geometric origin. 
        \item \label{KtopG-relation}
        In general, there does not seem to be a simple formula for $\Ktop[0](\cC^G/S) \otimes \bQ$ in terms of $\Ktop[0](\cC) \otimes \bQ$,  so it may be difficult to verify the existence of the required section $\wtilde{v}$.
    \end{enumerate}
Regarding~\eqref{CG-smooth-proper}, we expect the assumption that $\cC^G$ is of geometric origin is not essential in Theorem~\ref{theorem-semiregularity-categorical-equivariant}, but removing this hypothesis would require some other unknown conjectures about the Hodge theory of smooth proper categories. In this sense, we view issue~\eqref{CG-smooth-proper} as one of a technical nature. 
Moreover, as we have explained in \S\ref{section-geometrization}, in many interesting cases it is possible to verify that $\cC^G$ is of geometric origin.  

Regarding~\eqref{KtopG-relation}, 
by Lemma~\ref{lemma-Ktop-G} for abelian groups $G$ there is always an identification between the invariant parts $\Ktop[0](\cC^G/S)^{G^{\vee}} \otimes \bQ$ and $\Ktop[0](\cC/S)^G \otimes \bQ$. 
Moreover, by Lemma~\ref{lemma-G-action-Ktop-trivial} if the actions of $G^{\vee}$ and $G$ on $\cC^G$ and $\cC$ are through the identity components of the groups of autoequivalences over a point $0 \in S(\bC)$, then $\Ktop[0](\cC^G/S) \otimes \bQ$ and $\Ktop[0](\cC/S) \otimes \bQ$ are themselves identified. 
In this situation, the existence of the required section $\wtilde{v}$ is automatic.  

Theorem~\ref{theorem-semiregularity-equivariant-intro} proved below is one important situation where both difficulties~\eqref{CG-smooth-proper} and~\eqref{KtopG-relation} can be overcome along the above lines. 
\end{remark}

\begin{remark}
    There is a variant of Theorem~\ref{theorem-semiregularity-categorical-equivariant} where we consider sections of $\Ktop[0](\cC/S)$ and $\Ktop[0](\cC^G/S)$ instead of their rationalizations. 
    We formulated the theorem with rational coefficients, since it is more natural for applications. 
    Indeed, in Remark~\ref{remark-equivariant-semiregulary-conditions} we discussed a criterion for the existence of the section~$\wtilde{v}$ required in Theorem~\ref{theorem-semiregularity-categorical-equivariant}, but this criterion only ensures the existence of $\wtilde{v}$ as a section of $\Ktop[0](\cC^G/S) \otimes \bQ$. 
\end{remark}

\subsection{Twisted derived categories} 
Finally, we prove the semiregularity theorems for twisted derived categories promised in the introduction. 

\begin{proof}[Proof of Theorem~\ref{theorem-semiregularity-intro}]
    Write $B_0 = b_0/n$ for some $b_0 \in \rH^2(X, \bZ(1))$ and $n \in \bZ$. 
    Let $\theta_0$ be the image of $b_0$ under the map
    \begin{equation*}
            \exp(-/n) \colon \rH^2(X_0, \bZ(1)) \to \rH^2(X_0^{\an}, \bmu_n) = \rH^2_{\et}(X_0, \bmu_n). 
    \end{equation*}
    By the (proof of)~\cite[Lemma~3.5]{dJP-pi}, up to passing to an \'{e}tale neighborhood of $0 \in S(\bC)$, we may assume that there is an element $\theta \in \rH^2_{\et}(X, \bmu_n)$ whose restriction to the fiber $X_0$ is equal to $\theta_0$. 
    Let $\alpha \in \Br(X)$ be the Brauer class determined by $\theta$. 
    Note that $\alpha_0 = 0$ by the assumption that $B_0 \in \rH^2(X_0, \bQ(1))$ is algebraic. 
    Moreover, in view of the assumption that $w_0 = \exp(B_0) \cdot \ch(E_0)$ remains Hodge along $S$, it follows from Remark~\ref{remark-trivial-brauer-class} and Theorem~\ref{theorem-mukai-variation}\eqref{mukai-variation-section} that the class $v_{0} \in \Ktop[0](\Dperf(X_0, \alpha_0)) \otimes \bQ$ of $E_0$ lifts to a global section $v$ of the local system $\Ktop[0](\Dperf(X,\alpha)/S)\otimes \bQ$. 
    Now the result is a consequence Theorem~\ref{theorem-semiregularity-categorical} (in the form of Remark~\ref{remark-semiregularity-categorical-rational}). 
\end{proof}

\begin{proof}[Proof of Theorem~\ref{theorem-semiregularity-equivariant-intro}]
As in the proof of Theorem~\ref{theorem-semiregularity-intro} above, up to passing to an \'{e}tale neighborhood of $0 \in S(\bC)$, 
we may assume there is a $\bmu_n$-gerbe $\cX \to X$ (classified by the element $\theta \in \rH^2_{\et}(X, \bmu_n)$ above) such that: 
\begin{enumerate}
    \item The fiber $\cX_0 \to X_0$ is essentially trivial, i.e. if $\alpha \in \Br(X)$ denotes the Brauer class associated to $\cX \to X$, then the restriction $\alpha_0 = 0 \in \Br(X_0)$ vanishes.
    \item \label{section-v-proof-equivariant} The class $v_{0} \in \Ktop[0](\Dperf(X_0, \alpha_0)) \otimes \bQ$ of $E_0$ lifts to a global section $v$ of the local system $\Ktop[0](\Dperf(X,\alpha)/S)\otimes \bQ$. 
\end{enumerate} 

By shrinking $S$ if necessary, we may assume that $\omega_{X/S}$ is the pullback of a line bundle on~$S$. 
Then by Corollary~\ref{corollary-lifting-actions-in-CY-family}, up to possibly replacing $\cX$ with a suitable pushout and passing to an \'{e}tale neighborhood of $0 \in S(\bC)$, there exists: 
\begin{enumerate}[resume]
    \item \label{G-action-deform-DX0-proof}
    A $G$-action on the $S$-linear category $\Dperf(X,\alpha)$ whose base change to $0 \in S(\bC)$ recovers the given $G$-action on $\Dperf(X_0, \alpha_0) = \Dperf(X_0)$. 
    \item A finite group $G'$, a surjection $\pi \colon G' \twoheadrightarrow G$, and a $G'$-action on $\cX$ by gerbe automorphisms such that:
    \begin{enumerate}
        \item the induced $G'$-action on $\Dperf(X,\alpha)$ via Construction~\ref{remark-Aut-bmun-act-DXalpha} recovers the $G'$-action on $\Dperf(X, \alpha)$ obtained from the $G$-action in~\eqref{G-action-deform-DX0} by extension along the homomorphism $\pi \colon G' \to G$; and 
    \item \label{image-in-Autn} 
    the induced homomorphism $G' \to \Aut_{X/S}(S)$ has image contained in the subgroup $\Aut_{X/S}^0[n](S)$. 
    \end{enumerate}
\end{enumerate}

Note that since $X/S$ is an abelian scheme, it follows from condition~\eqref{image-in-Autn} that the image $H$ of $G'$ in $\Aut_{X/S}(S)$ acts freely on $X$. 
Thus we may apply Proposition~\ref{lemma-DperfXalphaG-geometric}\eqref{DXalphaG-geometric-origin} (with $G$ taken to be our $G'$ and $\bar{G}$ our $G$) to conclude that the invariant category $\Dperf(X,\alpha)^G$ is smooth and proper of geometric origin over $S$. 

To finish the proof, we would like to apply Theorem~\ref{theorem-semiregularity-categorical-equivariant} to $\cC = \Dperf(X, \alpha)$ with the $G$-action from~\eqref{G-action-deform-DX0-proof} above, the section $v$ of $\Ktop[0](\Dperf(X,\alpha)/S) \otimes \bQ$ from~\eqref{section-v-proof-equivariant}, 
and the given gluable weakly $G$-semiregular object $E_0 \in \Dperf(X_0, \alpha_0) = \Dperf(X_0)$. 
The only remaining hypothesis to verify is that if $\wtilde{E}_0 \in \Dperf(X_0)^G$ is a semiregular object such that $\Forg(\wtilde{E}_0) \simeq E_0$, 
then the class of $\wtilde{E}_0$ extends to section $\wtilde{v}$ of $\Ktop[0](\cC^G/S)\otimes \bQ$.  
As explained in Remark~\ref{remark-equivariant-semiregulary-conditions}, 
it suffices to show that $G^{\vee}$ and $G$ act on $\Dperf(X_0)^G$ and $\Dperf(X_0)$ through the identity components of the groups of autoequivalences. 

For the $G$-action on $\Dperf(X_0)$, this is true since by assumption $G \subset X_0 \times X_0^{\vee}$.  
For the $G^{\vee}$-action on $\Dperf(X_0)^G$, we first describe the action explicitly. 
Let $H_0 \subset \Aut_{X_0}^0(\bC) = X_0(\bC)$ denote the image of $G$. 
As noted above, $H_0$ acts freely on $X_0$, so that $X_0/H_0$ is an abelian variety. 
We claim that there exists an isogeny of abelian varieties $N \to X_0/H_0$  
and a Brauer class $\gamma \in \Br(N)$ such that: 
\begin{enumerate}[resume]
    \item \label{DX0-Ngamma}
    There is an equivalence $\Dperf(X_0)^G \simeq \Dperf(N, \gamma)$.
    \item Denoting by $\cN \to N$ the $\bG_m$-gerbe of class $\gamma$, there is a $G^{\vee}$-action on $\cN$ by gerbe automorphisms such that: 
    \begin{enumerate}
        \item the induced $G^{\vee}$-action on $\Dperf(N, \gamma)$ via Construction~\ref{remark-Aut-BGm-act-DXalpha} corresponds to the residual $G^{\vee}$-action on $\Dperf(X_0)^G$ under the equivalence from~\eqref{DX0-Ngamma}; and 
        \item \label{Gvee-act-through-Aut0N} the group $G^{\vee}$ acts on $\cN$ through the identity component $\Aut^0_{\cN}$, and hence on $\Dperf(X_0)^G$ through the identity component $\Aut_{\Dperf(X_0)^G}^0$. 
    \end{enumerate}
\end{enumerate}
To see this, we may apply Proposition~\ref{lemma-DperfXalphaG-geometric} to $\cX_0 \to X_0$ with its $G'$-action by gerbe automorphisms constructed above.  
By part~\eqref{M-decompose-Mi} of Proposition~\ref{lemma-DperfXalphaG-geometric}, 
$N \to X_0/H_0$ is a connected finite \'{e}tale cover, and hence an isogeny of abelian varieties. 
Note that torsion line bundles on $N$ are contained in $\Pic^0_N(\bC)$ and covering automorphisms of the isogeny $N \to X_0/H_0$ are given by translation by elements in the kernel, and hence are contained in $\Aut^0_N(\bC)$. 
Thus parts~\eqref{Gvee-act-cMi} and \eqref{Gvee-action-cMi-decomposition} of Proposition~\ref{lemma-DperfXalphaG-geometric} imply that $G^{\vee}$ does indeed act on $\cN$ through the identity component $\Aut_{\cN}^0$. 
\end{proof}

\begin{remark}
    In Theorem~\ref{theorem-semiregularity-equivariant-intro}, if we assume that $G \subset X_0$ acts only by translations, the proof above can be considerably simplified. 
    The point is that, in this case, the invariant category $\Dperf(X_0)^G$ is directly identified with the derived category of the abelian variety $X_0/G$. 
    For this reason, the use of the technical results from \S\ref{section-geometrization} may be circumvented.  
    However, the general version of Theorem~\ref{theorem-semiregularity-equivariant-intro} for arbitrary subgroups $G \subset X_0 \times X_0^{\vee}$ is much more convenient for applications. 
    For instance, as discussed in \S\ref{section-simplifying-markman}, 
    the more general situation arises naturally in Markman's approach to the Hodge conjecture for abelian varieties. 
\end{remark}


\newcommand{\etalchar}[1]{$^{#1}$}
\providecommand{\bysame}{\leavevmode\hbox to3em{\hrulefill}\thinspace}
\providecommand{\MR}{\relax\ifhmode\unskip\space\fi MR }
\providecommand{\MRhref}[2]{%
  \href{http://www.ams.org/mathscinet-getitem?mr=#1}{#2}
}
\providecommand{\href}[2]{#2}


\end{document}